\documentclass[11pt,a4paper,twoside]{article}

\usepackage{iftex}
\usepackage[T1]{fontenc}
\ifPDFTeX
  \usepackage[utf8]{inputenc}
\fi
\usepackage{lmodern}

\usepackage{amsmath,amsthm,amssymb,mathtools}

\usepackage{geometry}
\usepackage{xcolor}
\usepackage{microtype}
\usepackage{fancyhdr}

\usepackage{tikz}
\usetikzlibrary{arrows.meta,calc,positioning}
\usepackage[section]{placeins}
\usepackage{needspace}
\usepackage{caption}
\definecolor{figblue}{RGB}{41,85,125}
\definecolor{figorange}{RGB}{177,91,37}
\definecolor{figgreen}{RGB}{43,111,91}
\definecolor{figgray}{RGB}{93,101,109}
\tikzset{
  proof axis/.style={->,line width=.45pt,draw=figgray},
  proof arrow/.style={->,line width=.7pt,draw=figblue},
  proof curve/.style={line width=.85pt,draw=figblue},
  proof point/.style={circle,fill=figblue,inner sep=1.5pt},
  proof note/.style={align=center,inner sep=2pt},
  proof box/.style={draw=figblue!55,fill=figblue!3,rounded corners=2pt,
    align=center,inner sep=6pt,text width=3.5cm,minimum height=1cm}
}
\usepackage[
  colorlinks=true,
  linkcolor=blue!45!black,
  citecolor=blue!45!black,
  urlcolor=blue!45!black
]{hyperref}

\allowdisplaybreaks[3]
\newtheoremstyle{paperplain}{0.65em}{0.65em}{\itshape}{}%
  {\bfseries}{.}{0.5em}{}
\theoremstyle{paperplain}
\newtheorem{theorem}{Theorem}[section]
\newtheorem{proposition}[theorem]{Proposition}
\newtheorem{lemma}[theorem]{Lemma}
\newtheorem{corollary}[theorem]{Corollary}
\newtheoremstyle{paperremark}{0.65em}{0.65em}{\normalfont}{}%
  {\bfseries}{.}{0.5em}{}
\theoremstyle{paperremark}
\newtheorem{remark}[theorem]{Remark}
\newtheorem{example}[theorem]{Example}

\newcommand{\runningauthors}{Guohuan Qiu and Jin Yan}
\newcommand{\runningtitle}{}
\let\originaltitle\title
\renewcommand{\title}[2][]{%
  \originaltitle{#2}%
  \if\relax\detokenize{#1}\relax
    \gdef\runningtitle{#2}%
  \else
    \gdef\runningtitle{#1}%
  \fi
}

\numberwithin{equation}{section}
\numberwithin{figure}{section}
\newcommand{\R}{\mathbb R}
\newcommand{\Sym}{\operatorname{Sym}}
\newcommand{\rank}{\operatorname{rank}}
\newcommand{\diag}{\operatorname{diag}}
\newcommand{\dist}{\operatorname{dist}}
\newcommand{\ov}{\overline}
\newcommand{\norm}[1]{\left\|#1\right\|}
\newcommand{\cA}{\mathcal A}
\newcommand{\cL}{\mathcal L}
\title[Boundary regularity of $k$-Hessian eigenfunctions]{\textbf{Boundary regularity of $k$-Hessian eigenfunctions}}

\author{\textbf{Guohuan Qiu}\and\textbf{Jin Yan}}
\date{}

\makeatletter
\newcommand{\printfirstpagebottommatter}{%
  \begingroup
  \renewcommand{\thefootnote}{}%
  \renewcommand{\@makefntext}[1]{\noindent##1}%
  \footnotetext[0]{%
    \textit{Keywords.}
    $k$-Hessian equation; eigenfunction;
    boundary regularity; uniform ellipticity.
    \par\smallskip
    \textit{2020 Mathematics Subject Classification.}
    35J60, 35J70, 35P30.%
  }%
  \endgroup
}
\makeatother

\newcommand{\printauthorinformation}{%
  \par\bigskip
  \begingroup
  \small
  \setlength{\parindent}{0pt}%

  \noindent
  \textbf{Guohuan Qiu}\par
  State Key Laboratory of Mathematical Sciences,
  Academy of Mathematics and Systems Science,
  Chinese Academy of Sciences,
  Beijing 100190, China; \\
  Institute of Mathematics,
  Academy of Mathematics and Systems Science,
  Chinese Academy of Sciences,
  55 Zhongguancun East Road, Beijing 100190, China.\par
  \textit{Email:}
  \href{mailto:qiugh@amss.ac.cn}{qiugh@amss.ac.cn}.

  \par\medskip

  \noindent
  \textbf{Jin Yan}\par
  Institute of Mathematics,
  Academy of Mathematics and Systems Science,
  Chinese Academy of Sciences,
  55 Zhongguancun East Road, Beijing 100190, China.\par
  \textit{Email:}
  \href{mailto:yanjin@amss.ac.cn}{yanjin@amss.ac.cn}.

  \endgroup
}

\hypersetup{%
  pdftitle={Boundary regularity of k-Hessian eigenfunctions},
  pdfauthor={Guohuan Qiu and Jin Yan},
  pdfsubject={Boundary regularity of k-Hessian eigenfunctions},
  pdfkeywords={k-Hessian equation, eigenfunction, boundary regularity, uniform ellipticity}
}

\begin{document}

\maketitle
\printfirstpagebottommatter
\begin{abstract}
We prove smoothness up to the boundary for $k$-Hessian eigenfunctions on
smooth bounded uniformly convex domains when $2\leq k<n$.
This resolves an open problem in the regularity theory of degenerate
fully nonlinear elliptic equations and extends the boundary
smoothness theorem for Monge--Amp\`ere eigenfunctions. In contrast to the
Monge--Amp\`ere case, the linearized $k$-Hessian operator remains uniformly
elliptic up to the boundary.
\end{abstract}

\section{Introduction}

Let $\sigma_j$ denote the unnormalized elementary symmetric polynomial,
with $\sigma_0=1$, and let
\[
\Gamma_k=\{\lambda\in\R^n:\sigma_j(\lambda)>0,\ 1\leq j\leq k\}.
\]
We use the same notation for the corresponding cone of real symmetric
matrices. Let $\Omega\subset\R^n$ be a smooth bounded uniformly convex
domain. We consider the Hessian eigenvalue problem
\begin{equation}\label{eq:eigen}
\begin{cases}
\sigma_k(D^2u)=\lambda(-u)^k&\text{in }\Omega,\\
u=0&\text{on }\partial\Omega,
\end{cases}
\end{equation}
Here $u<0$, $D^2u\in\Gamma_k$ in $\Omega$, and $\lambda>0$.
The coefficient $\lambda$ is the $k$th power of the eigenvalue when the
operator is written as $\sigma_k^{1/k}$.

For $k=n$, equation \eqref{eq:eigen} is the Monge--Amp\`ere eigenvalue
equation. Lions \cite{Lions} proved the existence and uniqueness of the
eigenvalue and the uniqueness of its convex eigenfunction up to a positive
multiplicative constant. The eigenfunction is smooth in the interior and
belongs to $C^{1,1}(\ov\Omega)$. Tso \cite{Tso} obtained a variational
characterization of the eigenvalue; see also \cite{Le}.
Since the right-hand side vanishes on the boundary, higher boundary
regularity requires a separate argument.

Hong, Huang and Wang \cite{HHW} proved global smoothness in dimension
two. They used the curvature of level sets to control a nonvanishing
Hessian direction and overcome the boundary degeneracy. Savin developed
a localization theory for the boundary sections of degenerate
Monge--Amp\`ere equations. This geometric analysis led to boundary $C^2$
estimates and global $C^2$ regularity of the eigenfunctions in every
dimension \cite[Theorem~1.3]{SavinBoundary}.
Building on these estimates, Le and Savin established boundary Schauder
estimates and global $C^{2,\beta}$ regularity of the eigenfunctions
\cite[Theorems~1.1--1.3]{LS}. They then used hodograph and partial
Legendre transformations, together with Schauder estimates for
Grushin-type equations, to prove global smoothness
\cite[Theorem~1.4]{LS}. This resolved the long-standing boundary
smoothness problem for Monge--Amp\`ere eigenfunctions in all dimensions.

Wang \cite{Wang94} extended the eigenvalue theory to the $k$-Hessian
operator. On a smooth strictly $(k-1)$-convex domain, the eigenvalue is
positive and unique, and the admissible eigenfunction belongs to
$C^{\infty}(\Omega)\cap C^{1,1}(\ov\Omega)$ and is unique up to a
positive multiplicative constant. The precise statement appears in
\cite[Theorem~6.2, p.~211]{Wang09}, whose proof on the same page also states
interior smoothness. Le \cite{Le} obtained a spectral characterization of
the Hessian eigenvalue and studied a nondegenerate inverse iteration.
More recently, Zhang \cite{Zhang} established existence, uniqueness, and
$C^{\infty}(\Omega)\cap C^{1,1}(\ov\Omega)$ regularity for a class of
concave elliptic Hessian eigenvalue problems, including the $k$-Hessian
equation. These results provide the classical eigenfunction used below.

This boundary problem is part of the regularity theory of degenerate
Hessian equations. Ivochkina, Trudinger and Wang \cite{ITW}
obtained global second derivative estimates independent of a positive
lower bound for the right-hand side. Wang and Xu \cite{WX} established
$C^{1,1}$ solvability under weaker structural conditions, and Jiao and
Wang \cite{JW} obtained second order estimates for convex solutions
whose right-hand side $f$ satisfies $f^{1/(k-1)}\in C^{1,1}$.
These works describe the natural $C^{1,1}$ regularity for degenerate
Hessian equations. The additional structure of the eigenvalue equation
allows us to prove higher regularity.

The remaining question is whether the eigenfunction is smooth up to the
boundary when $1<k<n$. This is the natural extension of the global
smoothness problem for Monge--Amp\`ere eigenfunctions to the $k$-Hessian
equation. Le \cite[slide~32]{LeICBS} and He and Huang \cite[p.~3]{HeHuang}
explicitly stated this as an open problem.
We resolve this problem on smooth bounded uniformly convex domains.
In addition to smoothness, we prove that the linearized operator
has a positive lower bound up to the boundary, although
$\sigma_k(D^2u)$ vanishes there.

We write
\begin{equation}\label{eq:Fnotation}
F(M)=\sigma_k(M),\qquad
F^{ij}(M)=\frac{\partial F}{\partial M_{ij}}(M).
\end{equation}
For symmetric matrices $M$ and $\bar M$, this convention means
\[
\left.\frac{d}{ds}F(M+s\bar M)\right|_{s=0}
=\sum_{i,j=1}^nF^{ij}(M)\bar M_{ij}.
\]
Thus $(F^{ij})$ is the Newton tensor of order $k-1$.

\begin{theorem}\label{thm:main}
Let $2\leq k<n$, and let $\Omega\subset\R^n$ be a smooth bounded
uniformly convex domain. Suppose that
\[
u\in C^{\infty}(\Omega)\cap C^{1,1}(\ov\Omega)
\]
solves \eqref{eq:eigen}, and normalize $\norm u_{L^\infty(\Omega)}=1$.
Then there is $\theta>0$ such that
\begin{equation}\label{eq:gap}
\sum_{i,j=1}^nF^{ij}(D^2u(x))\xi^i\xi^j\geq\theta|\xi|^2
\qquad(x\in\Omega,\ \xi\in\R^n).
\end{equation}
Moreover, $u\in C^{\infty}(\ov\Omega)$.
\end{theorem}

The constant $\theta$ may depend on the fixed eigenfunction and the
domain. By the assumed $C^{1,1}$ bound, $(F^{ij})$ is bounded from above. Thus \eqref{eq:gap} states that the polynomial operator
$\sigma_k$ is uniformly elliptic along this solution.
Our assumption on the domain is stronger than the
$(k-1)$-convexity required in the existence theory. We do not assume that
$u$ is convex.

Tian, Wang and Xu \cite{TWX} constructed smooth local solutions with
uniformly elliptic linearization even when the right-hand side vanishes
or changes sign. Here we prove nondegeneracy along a global eigenfunction.
The following examples explain the difference between $k<n$ and $k=n$.

\begin{example}\label{ex:zero-level}
Let $k=2$, $n=3$, and
\[
M_0=\diag(1,1,-1/2).
\]
Then
\[
\sigma_1(M_0)=\frac32>0,\qquad
\sigma_2(M_0)=1-\frac12-\frac12=0.
\]
Thus $M_0\in\ov\Gamma_2\cap\Gamma_1$, while its linearization satisfies
\begin{equation}\label{eq:example-linearization}
(F^{ij}(M_0))
=\sigma_1(M_0)I-M_0
=\diag(1/2,1/2,2)>0.
\end{equation}
For comparison, $\diag(1,1,0)$ has zero determinant and the degenerate
cofactor matrix $\diag(0,0,1)$.
\end{example}

The next example describes how degeneracy can occur and illustrates
Lemma~\ref{lem:newton-limits}\textup{(i)}.

\begin{example}\label{ex:newton-degeneracy}
Let $k=2$ and $n=3$. Then
\[
F^{ij}(M)=\sigma_1(M)\delta_{ij}-M_{ij}.
\]
If $M\in\ov\Gamma_2$ and $\sigma_1(M)=0$, its eigenvalues satisfy
\[
0\leq2\sigma_2(M)=-\sum_{i=1}^3\lambda_i^2.
\]
Thus $M=0$, as asserted in Lemma~\ref{lem:newton-limits}\textup{(i)}.
If instead $(F^{ij}(M))$ is nonzero and degenerate, we may diagonalize
$M$ and arrange that $\lambda_1=\sigma_1(M)$. Hence
$\lambda_2+\lambda_3=0$ and
\[
0\leq\sigma_2(M)=\lambda_2\lambda_3=-\lambda_2^2.
\]
It follows that, in these coordinates,
\[
M=\diag(a,0,0),\qquad
(F^{ij}(M))=\diag(0,a,a),\qquad a>0.
\]
In particular, $M\geq0$, $\rank M=1$, and
$\sum_{h=1}^3M_{ih}F^{h j}(M)=0$, as in
Lemma~\ref{lem:newton-limits}\textup{(i)}. The kernel of $M$ is two-dimensional,
although the kernel of its linearization is only one-dimensional.
\end{example}

The next example shows how uniform boundary convexity excludes
degeneracy when the boundary Hessian exists.
\begin{example}\label{ex:convex-kernel}
Let $k=2$, $n=3$, and suppose that a boundary Hessian
$M\in\ov\Gamma_2$ exists at a uniformly convex boundary point.
The zero boundary condition and the Hopf lemma imply that the tangential
block $B$ is positive definite. Hence
\[
\xi=(\xi',0)\in\ker M
\quad\Longrightarrow\quad
0=\sum_{i,j=1}^2B_{ij}\xi^i\xi^j
\quad\Longrightarrow\quad \xi=0,
\]
and therefore $\dim\ker M\leq1$.
Moreover, $F^{33}(M)=M_{11}+M_{22}>0$. By
Example~\ref{ex:newton-degeneracy}, degeneracy of $(F^{ij}(M))$
would require $\dim\ker M=2$, a contradiction.
Thus $(F^{ij}(M))$ is positive definite.
\end{example}

The next example shows that the full conclusion of
Theorem~\ref{thm:main} does not extend to all strictly mean convex domains.

\begin{example}\label{ex:boundary-geometry}
Let $k=2$, $n=3$, and consider
\[
\Omega=\{x\in\R^3:x_1^2+x_2^4+x_3^2<1\}.
\]
This domain is smooth and bounded. Its principal curvatures are
positive when $x_2\ne0$, and are $1$ and $0$ when $x_2=0$.
Thus $\partial\Omega$ is strictly mean convex.

Let $u$ be the normalized negative eigenfunction. Uniqueness
\cite{Wang09} implies that $u$ is even in $x_1$ and $x_2$.
If Theorem~\ref{thm:main} held on this domain, $u$ would be smooth
up to the boundary. At $q=(0,0,1)$, symmetry and the zero boundary
condition would imply
\[
D^2u(q)=\diag(a,0,c),\qquad a=u_\nu(q)>0.
\]
The equation implies $ac=\sigma_2(D^2u(q))=0$, so $c=0$ and
\[
(F^{ij}(D^2u(q)))=\diag(0,a,a).
\]
This contradicts the uniform ellipticity asserted in
Theorem~\ref{thm:main}.
\end{example}

For $2\leq k<n$, a positive definite tangential block implies
$\dim\ker M\leq1$. By
Lemma~\ref{lem:newton-limits}\textup{(i)}, degeneracy of
$(F^{ij}(M))$ would require
\[
\dim\ker M\geq n-k+1\geq2
\]
and hence contradict this bound. Thus the kernel argument in
Example~\ref{ex:convex-kernel} extends to this range.
For $k=n$, the lower bound is only one, and there is no contradiction.

The main analytical difficulty is that we start with only
$C^{1,1}(\ov\Omega)$ regularity. Boundary convexity ensures positivity
of the tangential quadratic data, but does not provide a boundary trace
of $D^2u$. Thus formally differentiating the boundary condition twice
does not identify the tangential block of an interior Hessian limit.
We must prove that these limits preserve tangential positivity. The main
step is to classify the half-space limits. For positive constant
right-hand sides, Jia and Ma \cite{JiaMa} proved half-space Liouville
theorems under quadratic growth and additional structural assumptions.
Our limits have zero right-hand side, so their classification requires
a separate argument.

We study the tangential gradient error divided by the normal variable.
The rank and kernel properties of Newton tensors imply a coercivity
estimate after normalization, even when their traces tend to zero. The barrier
$1/t$ leads to an oblique boundary inequality for the squared gradient
quotient. This inequality rules out both boundary maxima and boundary
layers when the functions converge to a quadratic polynomial. We then use
a maximum argument and dilate maximizing sequences that escape to
infinity to show that every half-space limit arising from the equation
is a quadratic polynomial.

Convergence to a quadratic polynomial alone does not imply convergence
of Hessians. Savin's small perturbation theorem
\cite[Theorem~1.3, p.~560]{Savin} provides the main regularity step in
this passage. It requires uniform ellipticity only near the reference
Hessian, so we can apply it without already knowing that the Hessians
of the approximating solutions are close to that matrix.
An exact local comparison solution and a monotone extension verify the
hypotheses of the theorem. Its $C^{2,\beta}$ estimate then upgrades
uniform convergence to Hessian convergence and proves \eqref{eq:gap}.
Fan's nonhomogeneous extension \cite[Theorem~1.7, p.~3]{Fan} provides
an alternative application of the same perturbation argument.
A concave uniformly elliptic extension then yields boundary
$C^{2,\alpha}$ regularity, and boundary Schauder estimates control
all higher derivatives.

The regularity result also provides boundary information useful for
geometric inequalities for Hessian eigenvalues. Salani \cite{SalaniMA}
proved a Brunn--Minkowski inequality for the Monge--Amp\`ere eigenvalue, and
related results for the $2$-Hessian equation in dimension three were
obtained by Liu, Ma and Xu \cite{LMX} and Salani \cite{SalaniHessian}.
Recently, Li, Ma, Qiu and Salani \cite{LMQS} proved strict convexity of
$-\log(-u)$ and a Brunn--Minkowski inequality for the $2$-Hessian
eigenvalue in higher dimensions. Theorem~\ref{thm:main} improves the
boundary regularity of the eigenfunctions considered there. It also
implies a direct quantitative lower bound for $D^2[-\log(-u)]$ near the
boundary; see Corollary~\ref{cor:log-collar}. The global convexity
conclusion in \cite{LMQS} uses the interior arguments developed in
that paper.

Section~\ref{sec:algebra} collects the matrix identities and classifies
quadratic solutions of the equation with zero right-hand side and positive tangential
data. We begin the proof of Theorem~\ref{thm:main} in
Section~\ref{sec:gap}: we construct and classify the boundary limits,
prove Hessian convergence, and establish \eqref{eq:gap}. We complete the
proof in Section~\ref{sec:regularity} by establishing higher boundary
regularity.

\section{Algebraic preliminaries}\label{sec:algebra}

We first describe how the Newton tensor and its normalized limits can
degenerate. We then identify the quadratic models with positive
tangential data and prove that their linearizations are positive definite.
Constants $c,C$ may change from line to line.
Superscripts $l$ index sequences.

For a matrix $M$, we use the norm
$\norm M=(\sum_{i,j}M_{ij}^2)^{1/2}$.
We use the Euclidean form of the Newton transformation in
\cite{QZ}:
\begin{equation}\label{eq:Newton-definition}
[T_j(M)]^{ab}
=\frac1{j!}\sum_{i_1,\ldots,i_j=1}^n
\sum_{h_1,\ldots,h_j=1}^n
\delta^{b i_1\cdots i_j}_{a h_1\cdots h_j}
M_{i_1 h_1}\cdots M_{i_j h_j}.
\end{equation}
Here $[T_0(M)]^{ab}=\delta_{ab}$. The generalized Kronecker symbol
$\delta^{i_0\cdots i_j}_{h_0\cdots h_j}$ is the determinant of
the matrix $(\delta_{i_p h_q})_{p,q=0}^j$.
We identify the two tensor forms using the Euclidean metric. In particular,
$F^{ij}(M)=[T_{k-1}(M)]^{ij}$.
We use the same definition in dimension $n-1$ for tangential matrices.
The trace and Euler identities are
\begin{equation}\label{eq:Newton}
\begin{aligned}
\sum_{i=1}^nF^{ii}(M)&=(n-k+1)\sigma_{k-1}(M),\\
\sum_{i,j=1}^nF^{ij}(M)M_{ij}&=k\sigma_k(M).
\end{aligned}
\end{equation}
We use the standard properties of the G\aa rding cone: its convexity,
its preservation under addition of positive semidefinite matrices,
the concavity of $\sigma_k^{1/k}$, and the Newton--Maclaurin
inequalities; see \cite{Wang09}. The supporting inequality states that,
for $M,\bar M\in\Gamma_k$,
\begin{equation}\label{eq:support}
\sum_{i,j=1}^n F^{ij}(M)\bar M_{ij}
\geq k\sigma_k(M)^{(k-1)/k}\sigma_k(\bar M)^{1/k}.
\end{equation}
The closed cone is $\ov\Gamma_k=\{\sigma_j\geq0:1\leq j\leq k\}$.
See also \cite{IPY} for the relation between G\aa rding cones and
ellipticity of Hessian operators.

In a diagonal basis, let $\mu$ be obtained by deleting the
$i$th eigenvalue of $M\in\ov\Gamma_k$. Then
\[
\begin{aligned}
\sigma_j(M+s e_i e_i^T)
&=\sigma_j(M)+s\sigma_{j-1}(\mu)\geq0,
\qquad s\geq0,\ 1\leq j\leq k,\\
\sigma_{j-1}(\mu)
&=\lim_{s\to\infty}\frac{\sigma_j(M+s e_i e_i^T)}s\geq0.
\end{aligned}
\]
Thus $\mu\in\ov\Gamma_{k-1}$ and $(F^{ij}(M))\geq0$.

The nonzero degenerate case in part \textup{(i)} of the next lemma is
the matrix form of the classification of Tian, Wang and Xu
\cite[Theorem~2.5]{TWX}. We include its proof and then treat normalized
Newton tensors along convergent matrix sequences, including when
their traces tend to zero.

\begin{lemma}\label{lem:newton-limits}
The following assertions hold for $2\leq k\leq n$.
\begin{enumerate}
\renewcommand{\labelenumi}{\textup{(\roman{enumi})}}
\item If $M\in\ov\Gamma_k$ and $\sigma_{k-1}(M)=0$, then
$M\geq0$ and $\rank M\leq k-2$.

Assume $k<n$. If $M\in\ov\Gamma_k$ and $(F^{ij}(M))$ is nonzero
and degenerate, then $M\geq0$, $\rank M=k-1$, and
\[
\sum_{h=1}^n M_{ih}F^{h j}(M)=0
\qquad(1\leq i,j\leq n).
\]

\item Suppose that $M^l\in\ov\Gamma_k$, $M^l\to M$,
$\sum_{h=1}^nF^{hh}(M^l)>0$, and
\begin{equation}\label{eq:normalized-limit}
\frac{F^{ij}(M^l)}{\sum_{r=1}^nF^{rr}(M^l)}
\longrightarrow\bar F^{ij}
\qquad(1\leq i,j\leq n).
\end{equation}
Then $(\bar F^{ij})\geq0$, $\sum_{i=1}^n\bar F^{ii}=1$, and
\begin{equation}\label{eq:rank-bound}
\rank(\bar F^{ij})\geq n-k+1.
\end{equation}
If $k<n$ and $(\bar F^{ij})$ is degenerate, then $M\geq0$ and
\begin{equation}\label{eq:normalized-kernel}
\sum_{h=1}^n M_{ih}\bar F^{h j}=0
\qquad(1\leq i,j\leq n).
\end{equation}

\end{enumerate}
\end{lemma}

\begin{proof}
For the first assertion in \textup{(i)}, we argue by induction on $k$.
Write $\lambda=(\lambda_1,\ldots,\lambda_n)$ for the eigenvalues of
$M$, and let $\lambda|i$ denote the vector obtained by deleting $\lambda_i$.
When $k=2$,
\[
0\leq\sum_{i=1}^n\lambda_i^2
=\sigma_1(M)^2-2\sigma_2(M)
=-2\sigma_2(M)\leq0,
\]
so $M=0$.

For $k\geq3$, the deletion property implies
$\sigma_{k-2}(\lambda|i)\geq0$, and
\[
\begin{aligned}
0\leq\sum_{i=1}^n\lambda_i^2\sigma_{k-2}(\lambda|i)
&=\sum_{i=1}^n\lambda_i
\bigl(\sigma_{k-1}(M)-\sigma_{k-1}(\lambda|i)\bigr)\\
&=\sigma_1(M)\sigma_{k-1}(M)-k\sigma_k(M)\\
&=-k\sigma_k(M)\leq0.
\end{aligned}
\]
Hence $\lambda_i^2\sigma_{k-2}(\lambda|i)=0$ for every $i$.
If $M\ne0$, then $\sigma_1(M)\geq0$ allows us to choose
$\lambda_1>0$. By the induction hypothesis,
\[
\begin{gathered}
\lambda|1\in\ov\Gamma_{k-1},\qquad
\sigma_{k-2}(\lambda|1)=0,\\
\lambda|1\geq0,\qquad \#\{i\geq2:\lambda_i>0\}\leq k-3.
\end{gathered}
\]
Thus $M\geq0$ and $\rank M\leq k-2$.

Now let $k<n$ and suppose $(F^{ij}(M))$ is nonzero and degenerate.
In a common diagonal basis, arrange
$F^{11}(M)=\sigma_{k-1}(\lambda|1)=0$. Then
\[
\begin{gathered}
\lambda|1\in\ov\Gamma_{k-1},\\
\sigma_k(\lambda|1)
=\sigma_k(M)-\lambda_1\sigma_{k-1}(\lambda|1)
=\sigma_k(M)\geq0.
\end{gathered}
\]
Since $n-1\geq k$, the first assertion applies in dimension $n-1$:
\[
\lambda|1\geq0,\qquad \#\{i\geq2:\lambda_i>0\}\leq k-2.
\]
By \eqref{eq:Newton},
\[
0<\sigma_{k-1}(M)
=\lambda_1\sigma_{k-2}(\lambda|1),
\]
so $\lambda_1>0$, exactly $k-2$ of $\lambda_2,\ldots,\lambda_n$
are positive, and
\[
M\geq0,\qquad \rank M=k-1,\qquad
\lambda_jF^{jj}(M)
=\lambda_j\sigma_{k-1}(\lambda|j)=0.
\]
By orthogonal invariance,
$\sum_{h=1}^nM_{ih}F^{hj}(M)=0$ in the original coordinates.

For \textup{(ii)}, the normalized limit satisfies
\[
(\bar F^{ij})\geq0,\qquad \sum_{i=1}^n\bar F^{ii}=1.
\]
For each $\bar M\in\Gamma_k$, \eqref{eq:support} implies
\[
\begin{aligned}
\sum_{a,b=1}^nF^{ab}(M^l)\bar M_{ab}
&=\lim_{\delta\downarrow0}
\sum_{a,b=1}^nF^{ab}(M^l+\delta I)\bar M_{ab}\geq0,\\
\sum_{a,b=1}^n\bar F^{ab}\bar M_{ab}
&=\lim_{l\to\infty}
\frac{\sum_{a,b=1}^nF^{ab}(M^l)\bar M_{ab}}
{\sum_{h=1}^nF^{hh}(M^l)}\geq0.
\end{aligned}
\]
If $m:=\dim\ker(\bar F^{ij})\geq k$, choose a basis adapted to
this kernel and set
\[
\bar M=\diag\bigl(
\underbrace{1,\ldots,1}_{m},
\underbrace{-\varepsilon,\ldots,-\varepsilon}_{n-m}\bigr).
\]
For $1\leq j\leq k$ and small $\varepsilon>0$,
\[
\sigma_j(\bar M)=\binom mj+O(\varepsilon)>0,
\qquad
0\leq\sum_{a,b=1}^n\bar F^{ab}\bar M_{ab}
=-\varepsilon\sum_{i=1}^n\bar F^{ii}=-\varepsilon,
\]
a contradiction. Therefore $m\leq k-1$, proving \eqref{eq:rank-bound}.

Suppose next that $k<n$ and $(\bar F^{ij})$ is degenerate.
If $\sum_hF^{hh}(M)>0$, then
\[
\bar F^{ij}=\frac{F^{ij}(M)}{\sum_{h=1}^nF^{hh}(M)},
\]
and \textup{(i)} implies $M\geq0$ and \eqref{eq:normalized-kernel}.
Otherwise, \eqref{eq:Newton} and \textup{(i)} imply
\[
\sigma_{k-1}(M^l)\longrightarrow\sigma_{k-1}(M)=0,
\qquad M\geq0.
\]
By Newton--Maclaurin,
\[
0\leq\frac{\sigma_k(M^l)}{\sigma_{k-1}(M^l)}
\leq C\sigma_{k-1}(M^l)^{1/(k-1)}\longrightarrow0.
\]
Consequently,
\[
\begin{aligned}
\sum_{a,b=1}^nM_{ab}\bar F^{ab}
&=\lim_{l\to\infty}
\frac{\sum_{a,b=1}^nM^l_{ab}F^{ab}(M^l)}
{\sum_{h=1}^nF^{hh}(M^l)}\\
&=\frac{k}{n-k+1}\lim_{l\to\infty}
\frac{\sigma_k(M^l)}{\sigma_{k-1}(M^l)}=0.
\end{aligned}
\]
In an eigenbasis of $M$,
\[
\lambda_i\geq0,\qquad
\sum_{i=1}^n\lambda_i\bar F^{ii}=0,\qquad
|\bar F^{ij}|^2\leq\bar F^{ii}\bar F^{jj}.
\]
Thus $\lambda_i\bar F^{ij}=0$ for all $i,j$; returning to the
original coordinates proves \eqref{eq:normalized-kernel}.
\end{proof}

We next classify the quadratic solutions with positive tangential data.
Write $X=(y,t)\in\R_+^n:=\R^{n-1}\times(0,\infty)$ and
$\ov{\R_+^n}=\R^{n-1}\times[0,\infty)$.
For $2\leq k<n$, a positive definite $(n-1)\times(n-1)$ matrix $B$, and
$b\in\R^{n-1}$, set
\begin{equation}\label{eq:quadratics}
\begin{aligned}
c_b&=\frac{\displaystyle\sum_{i,j=1}^{n-1}[T_{k-2}(B)]^{ij}b^ib^j
-\sigma_k(B)}{\sigma_{k-1}(B)},\\[3pt]
\bar M_b&=\begin{pmatrix}B&b\\b^T&c_b\end{pmatrix},\\[3pt]
Q_{B,b}(y,t)&=\frac12\sum_{i,j=1}^{n-1}B_{ij}y^iy^j
+t\sum_{i=1}^{n-1}b^iy^i+\frac12c_b t^2.
\end{aligned}
\end{equation}

The positive definite linearizations of these quadratic models allow us
to use them in the perturbation argument in Section~\ref{sec:savin}.

\begin{lemma}\label{lem:quadratics}
Let $2\leq k<n$ and $B>0$. A quadratic polynomial $Q$ satisfies
\[
Q(y,0)=\frac12\sum_{i,j=1}^{n-1}B_{ij}y^i y^j,
\qquad DQ(0)=0,\qquad \sigma_k(D^2Q)=0
\]
if and only if $Q=Q_{B,b}$ for a unique $b\in\R^{n-1}$.
Its Hessian satisfies
\begin{equation}\label{eq:quadratic-positive}
\bar M_b\in\ov\Gamma_k,\qquad
\sigma_j(\bar M_b)>0\quad(1\leq j<k),\qquad
(F^{ij}(\bar M_b))>0.
\end{equation}
More precisely, if
\[
B\geq\kappa I_{n-1},\qquad
\bar M_b\leq L I_n,\qquad 0<\kappa\leq L,
\]
then
\begin{equation}\label{eq:quadratic-sigma-bound}
\sigma_j(\bar M_b)
\geq \frac{k-j}{k}\binom{n}{j}\kappa^j
\qquad(1\leq j<k),
\end{equation}
and
\begin{equation}\label{eq:quadratic-ellipticity}
\frac1k\binom{n-2}{k-1}\frac{\kappa^k}{L}\,I_n
\leq (F^{ij}(\bar M_b))
\leq \binom{n-1}{k-1}L^{k-1}I_n.
\end{equation}
\end{lemma}

\begin{proof}
The boundary data and $DQ(0)=0$ imply
\[
Q(y,t)=\frac12\sum_{i,j=1}^{n-1}B_{ij}y^iy^j
+t\sum_{i=1}^{n-1}b^iy^i+\frac c2t^2.
\]
Expansion by principal minors yields
\begin{equation}\label{eq:block}
\begin{aligned}
\sigma_k\begin{pmatrix}B&b\\b^T&c\end{pmatrix}
&=\sigma_k(B)+c\sigma_{k-1}(B)\\
&\quad-\sum_{i,j=1}^{n-1}[T_{k-2}(B)]^{ij}b^ib^j.
\end{aligned}
\end{equation}
Since $\sigma_{k-1}(B)>0$, the equation is equivalent to $c=c_b$.
Differentiating \eqref{eq:block}, we also obtain
\begin{equation}\label{eq:normal-column}
\begin{aligned}
F^{an}(\bar M_b)&=-\sum_{i=1}^{n-1}[T_{k-2}(B)]^{ai}b^i
\qquad(a<n),\\
F^{nn}(\bar M_b)&=\sigma_{k-1}(B).
\end{aligned}
\end{equation}

By interlacing, $\bar M_b$ has $n-1$ positive eigenvalues
$\mu_1,\ldots,\mu_{n-1}$. Since $k\leq n-1$ and
$\sigma_k(\bar M_b)=0$, the remaining eigenvalue is $-r<0$ and
\[
r=\frac{\sigma_k(\mu)}{\sigma_{k-1}(\mu)},\qquad
\mu=(\mu_1,\ldots,\mu_{n-1}).
\]
Under the stated bounds, interlacing and Newton's inequalities imply
\[
\kappa\leq\mu_i\leq L,\qquad
r\leq\frac{j(n-k)}{k(n-j)}
\frac{\sigma_j(\mu)}{\sigma_{j-1}(\mu)}
\quad(1\leq j<k).
\]
Hence
\[
\begin{aligned}
\sigma_j(\bar M_b)
&=\sigma_j(\mu)-r\sigma_{j-1}(\mu)\\
&\geq\frac{n(k-j)}{k(n-j)}\sigma_j(\mu)
\geq\frac{k-j}{k}\binom nj\kappa^j>0.
\end{aligned}
\]
This proves \eqref{eq:quadratic-sigma-bound} and
$\bar M_b\in\ov\Gamma_k\cap\Gamma_{k-1}$.

In the eigenbasis $(\mu_1,\ldots,\mu_{n-1},-r)$, write $\mu|i$
for the vector obtained by deleting $\mu_i$. For $i<n$, Newton's
inequality in dimension $n-2$ implies
\[
\begin{aligned}
F^{ii}(\bar M_b)
&=\sigma_{k-1}(\mu|i)-r\sigma_{k-2}(\mu|i)\\
&=\frac{\sigma_{k-1}(\mu|i)^2
-\sigma_k(\mu|i)\sigma_{k-2}(\mu|i)}{\sigma_{k-1}(\mu)}\\
&\geq\frac{n-1}{k(n-k)}
\frac{\sigma_{k-1}(\mu|i)^2}{\sigma_{k-1}(\mu)}.
\end{aligned}
\]
When $k=n-1$, the term $\sigma_k(\mu|i)$ is zero.
Since $\mu_j\geq\kappa$, counting monomials yields
\[
(k-1)\sigma_{k-1}(\mu|i)
\geq(n-k)\kappa\sigma_{k-2}(\mu|i).
\]
Therefore
\[
\begin{aligned}
\sigma_{k-1}(\mu)
&=\mu_i\sigma_{k-2}(\mu|i)+\sigma_{k-1}(\mu|i)\\
&\leq\frac{(n-1)L}{(n-k)\kappa}\sigma_{k-1}(\mu|i),\\
F^{ii}(\bar M_b)
&\geq\frac{\kappa}{kL}\sigma_{k-1}(\mu|i)
\geq\frac1k\binom{n-2}{k-1}\frac{\kappa^k}{L}.
\end{aligned}
\]
The upper bounds and the normal entry satisfy
\[
\begin{gathered}
F^{ii}(\bar M_b)\leq\sigma_{k-1}(\mu|i)
\leq\binom{n-2}{k-1}L^{k-1}\qquad(i<n),\\
\binom{n-1}{k-1}\kappa^{k-1}
\leq F^{nn}(\bar M_b)=\sigma_{k-1}(\mu)
\leq\binom{n-1}{k-1}L^{k-1}.
\end{gathered}
\]
These estimates prove \eqref{eq:quadratic-ellipticity}.
For arbitrary $B>0$ and $b$, choose $\kappa,L$ with
$B\geq\kappa I$ and $\bar M_b\leq LI$ to obtain
\eqref{eq:quadratic-positive}.
\end{proof}

\section{Proof of Theorem~\ref{thm:main}: uniform ellipticity}\label{sec:gap}

We begin the proof of Theorem~\ref{thm:main} by establishing \eqref{eq:gap}.
We argue by contradiction. A dilation at the boundary produces a
half-space limit. Classifying this limit and proving Hessian convergence
rule out degeneracy of the linearization. We prove smoothness in
Section~\ref{sec:regularity}.

\subsection{Boundary dilation}\label{sec:boundary-dilation}

Since $\Delta u>0$, $u<0$ in $\Omega$, and $u=0$ on $\partial\Omega$,
the Hopf lemma \cite[Lemma~3.4, p.~34]{GT} and compactness imply
\begin{equation}\label{eq:Hopf-bounds}
0<a_*\leq u_\nu\leq a^*<\infty\quad\text{on }\partial\Omega,
\end{equation}
where $\nu$ is the outward unit normal. Fix a bound
$\norm{D^2u}_{L^\infty(\Omega)}\leq K_0$.
If \eqref{eq:gap} fails, choose $x^l\in\Omega$ such that
\[
\lambda_{\min}\bigl((F^{ij}(D^2u(x^l)))\bigr)\longrightarrow0.
\]
By interior ellipticity, after passing to a subsequence,
\[
r^l:=\dist(x^l,\partial\Omega)\downarrow0.
\]
Let $q^l$ be the nearest boundary point. Translate $q^l$ to $0$
and rotate the inward normal to $e_n$. Near the origin, the domain lies above
\begin{equation}\label{eq:local-boundary-graph}
x^n=\rho^l(x'),\qquad \rho^l(0)=0,\qquad D\rho^l(0)=0,
\end{equation}
and $x^l=r^le_n$ in these coordinates. Denote the transformed
function by $u^l$ and set
\[
a^l=-\partial_nu^l(0)=u_\nu(q^l),\qquad
a_*\leq a^l\leq a^*,\qquad Du^l(0)=-a^le_n.
\]
Introduce
\begin{equation}\label{eq:dilated-coordinates}
X=(y,t)=\frac{x}{r^l},\qquad x'=r^ly,\qquad x^n=r^lt,
\end{equation}
and define
\begin{equation}\label{eq:boundary-dilation}
U^l(X)=\frac{u^l(r^lX)+a^lr^lt}{(r^l)^2}.
\end{equation}
Taylor's inequality and differentiation imply
\[
\begin{gathered}
|u^l(x)+a^lx^n|\leq\frac{K_0}{2}|x|^2,\qquad
|U^l(X)|\leq\frac{K_0}{2}|X|^2,\\
U^l(0)=0,\qquad D_XU^l(0)=0,\qquad
|D_XU^l(X)|\leq K_0|X|,
\end{gathered}
\]
and
\begin{equation}\label{eq:rescaled-hessian}
D_X^2U^l(X)=D_x^2u^l(r^lX),\qquad
\norm{D_X^2U^l}\leq K_0.
\end{equation}
The rescaled domain lies above
\begin{equation}\label{eq:rescaled-boundary}
t=\frac{\rho^l(r^ly)}{r^l}.
\end{equation}
On every bounded tangential set,
\[
\begin{gathered}
\frac{\rho^l(r^ly)}{r^l}=O(r^l|y|^2),\qquad
D_y\frac{\rho^l(r^ly)}{r^l}=O(r^l|y|),\\
D_y^2\frac{\rho^l(r^ly)}{r^l}=r^lD^2\rho^l(r^ly)=O(r^l).
\end{gathered}
\]
Thus the rescaled domains converge locally to $\R_+^n$.

On the rescaled boundary,
\[
U^l\left(y,\frac{\rho^l(r^ly)}{r^l}\right)
=\frac{a^l\rho^l(r^ly)}{(r^l)^2}.
\]
Uniform convexity, \eqref{eq:Hopf-bounds}, and Taylor's inequality imply
\[
\kappa I_{n-1}\leq a^lD^2\rho^l(0)\leq K_0I_{n-1}.
\]
After passing to a subsequence,
\begin{equation}\label{eq:boundary-data-limit}
a^lD^2\rho^l(0)\longrightarrow B,\qquad
\kappa I_{n-1}\leq B\leq K_0I_{n-1},
\end{equation}
and
\[
\frac{a^l\rho^l(r^ly)}{(r^l)^2}
=\frac12\sum_{i,j=1}^{n-1}a^l\partial_{ij}\rho^l(0)y^iy^j
+O(r^l|y|^3)
\longrightarrow\frac12\sum_{i,j=1}^{n-1}B_{ij}y^iy^j.
\]
Figure~\ref{fig:boundary-dilation} shows how the boundary height and
boundary value scale differently.

To obtain compactness up to $t=0$, consider the pullbacks
\[
(y,t)\longmapsto
U^l\left(y,t+\frac{\rho^l(r^ly)}{r^l}\right).
\]
The coordinate maps converge to the identity in $C^2$ locally.
The pullbacks have common local $C^{1,1}$ bounds and hence, after passing to
a subsequence, converge in $C^1_{\mathrm{loc}}(\ov{\R_+^n})$ to $U$.
On interior compact sets, the original functions $U^l$ have the same limit.
Consequently,
\[
\begin{gathered}
U(y,0)=\frac12\sum_{i,j=1}^{n-1}B_{ij}y^iy^j,\qquad
U(0)=0,\qquad DU(0)=0,\\
D^2U^l\overset{*}{\rightharpoonup}D^2U
\quad\text{in }L^\infty_{\mathrm{loc}}(\R_+^n),\qquad
\norm{D^2U}_{L^\infty(\R_+^n)}\leq K_0.
\end{gathered}
\]

\begin{figure}[htbp]
\centering
\begin{tikzpicture}[font=\footnotesize,>=Stealth]
\begin{scope}
  \node[proof note] at (0,3.6) {Translated and rotated coordinates};
  \fill[figblue!5] plot[domain=-2.05:2.05,samples=40] (\x,{.20*\x*\x})
    --(2.05,3.05)--(-2.05,3.05)--cycle;
  \draw[proof axis] (-2.15,0)--(2.2,0) node[right] {$x'$};
  \draw[proof axis] (0,-.35)--(0,3.13) node[right] {$x^n$};
  \draw[proof curve,domain=-2.05:2.05,samples=40]
    plot (\x,{.20*\x*\x});
  \node[proof point,label={below left:$q^l=0$}] at (0,0) {};
  \node[proof point,label={above right:$x^l=r^l e_n$}] at (0,1.6) {};
  \draw[<->,figorange] (.65,.04)--(.65,1.56)
    node[midway,right] {$r^l$};
  \draw[->,thick,figgreen] (0,.1)--(0,.85);
  \node[left,figgreen] at (0,.75) {$e_n$};
  \node[figblue] at (-1.25,2.25) {$\Omega$};
  \node[proof note] at (0,-.7) {$x^n=\rho^l(x')$};
\end{scope}
\begin{scope}[shift={(5,0)}]
  \node[proof note] at (0,3.6) {Dilation by $(r^l)^{-1}$};
  \fill[figblue!5] plot[domain=-2.05:2.05,samples=40] (\x,{.10*\x*\x})
    --(2.05,3.05)--(-2.05,3.05)--cycle;
  \draw[proof axis] (-2.15,0)--(2.2,0) node[right] {$y$};
  \draw[proof axis] (0,-.35)--(0,3.13) node[right] {$t$};
  \draw[proof curve,domain=-2.05:2.05,samples=40]
    plot (\x,{.10*\x*\x});
  \node[proof point,label={below left:$0$}] at (0,0) {};
  \node[proof point,label={above right:$e_n$}] at (0,1.6) {};
  \draw[<->,figorange] (.65,.04)--(.65,1.56)
    node[midway,right] {$1$};
  \node[proof note] at (0,-.7) {$t=(r^l)^{-1}\rho^l(r^l y)$};
\end{scope}
\begin{scope}[shift={(10,0)}]
  \node[proof note] at (0,3.6) {The local limit};
  \fill[figblue!5] (-2.05,0) rectangle (2.05,3.05);
  \draw[proof axis] (-2.15,0)--(2.2,0) node[right] {$y$};
  \draw[proof axis] (0,-.35)--(0,3.13) node[right] {$t$};
  \draw[figblue!30,domain=-2.05:2.05,samples=40]
    plot (\x,{.10*\x*\x});
  \draw[figblue!50,domain=-2.05:2.05,samples=40]
    plot (\x,{.045*\x*\x});
  \draw[figblue!65,domain=-2.05:2.05,samples=40]
    plot (\x,{.015*\x*\x});
  \draw[proof curve] (-2.05,0)--(2.05,0);
  \node[proof point,label={below left:$0$}] at (0,0) {};
  \node[proof point,label={above right:$e_n$}] at (0,1.6) {};
  \node[figblue] at (-1.15,2.25) {$\R_+^n$};
  \node[proof note] at (0,-.7) {$t=0$};
\end{scope}
\draw[figgray!35] (-2.1,-1.15)--(12.1,-1.15);
\node[proof note,text width=6.6cm,anchor=north] at (1.25,-1.30)
  {Boundary height\\[3pt]$\displaystyle\frac{\rho^l(r^l y)}{r^l}\longrightarrow0$};
\node[proof note,text width=7.4cm,anchor=north] at (8.5,-1.30)
  {Boundary value\\[3pt]$\displaystyle
    \frac{a^l\rho^l(r^l y)}{(r^l)^2}
       \longrightarrow\frac12\sum_{i,j=1}^{n-1}B_{ij}y^iy^j$};
\end{tikzpicture}
\caption{The boundary dilation \eqref{eq:boundary-dilation}, in a
two-dimensional normal section. The nearest boundary point becomes
the origin, and the chosen interior point becomes $e_n$.
The rescaled boundary becomes flat, while its quadratic contribution
to the boundary values remains. Adding $a^l r^l t$ removes
the first-order part of $u$.}
\label{fig:boundary-dilation}
\end{figure}
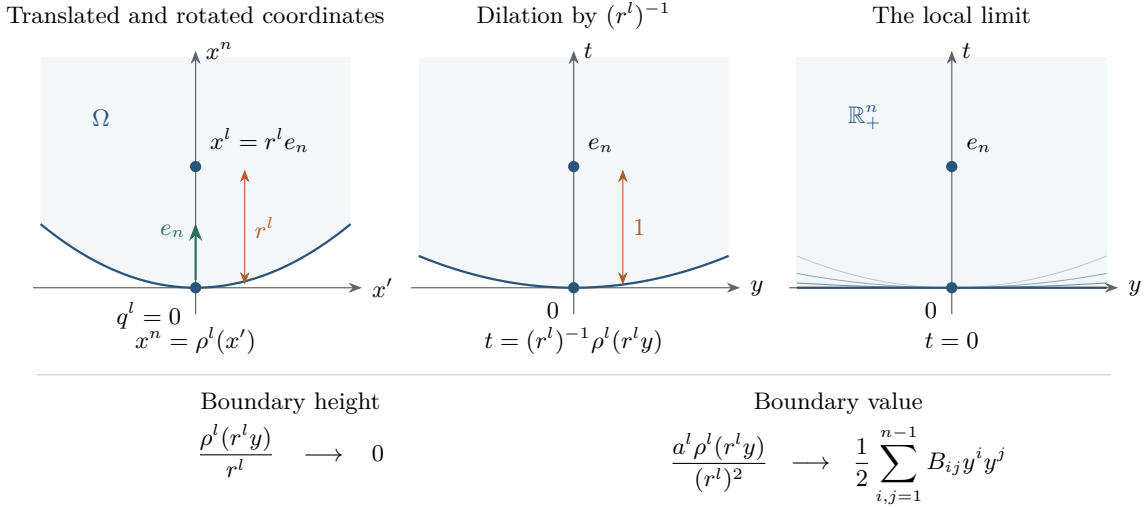

On every interior compact set, $D^2U^l\in\Gamma_k$ for large $l$, and
\begin{equation}\label{eq:blowup-source}
\begin{aligned}
f^l:=\sigma_k(D^2U^l)
&=\lambda\bigl(-u^l(r^lX)\bigr)^k\\
&=\lambda(r^l)^k(a^lt-r^lU^l)^k>0.
\end{aligned}
\end{equation}
Thus $\norm{f^l}_{L^\infty(K)}\leq C_K(r^l)^k\to0$ for
$K\Subset\R_+^n$. The Newton--Maclaurin inequality and differentiation imply
\[
\begin{aligned}
\sum_{i=1}^nF^{ii}(D^2U^l)
&=(n-k+1)\sigma_{k-1}(D^2U^l)\\
&\geq k\binom nk^{1/k}(f^l)^{(k-1)/k},\\
|D_Xf^l|
&\leq k\lambda r^l\bigl(-u^l(r^lX)\bigr)^{k-1}
\norm{Du}_{L^\infty(\Omega)}.
\end{aligned}
\]
Therefore
\begin{equation}\label{eq:normalized-source}
\frac{|D_Xf^l|}{\sum_{i=1}^nF^{ii}(D^2U^l)}
\leq\left(\frac{\lambda}{\binom nk}\right)^{1/k}
\norm{Du}_{L^\infty(\Omega)}r^l\longrightarrow0.
\end{equation}

To show that $U$ satisfies the limiting equation, let $X_0$ be a point
of second differentiability and set $M=D^2U(X_0)$. The perturbed Taylor polynomials
\[
P^\pm(X)=U(X_0)+DU(X_0)\cdot(X-X_0)
+\frac12\sum_{i,j=1}^n(M_{ij}\pm\varepsilon\delta_{ij})
(X^i-X_0^i)(X^j-X_0^j)
\]
touch $U$ strictly from above and below, respectively, on a small ball.
At nearby points of contact from above with $U^l$,
\[
D^2U^l\leq M+\varepsilon I,\qquad
D^2U^l\in\Gamma_k
\quad\Longrightarrow\quad M+\varepsilon I\in\Gamma_k.
\]
Letting $\varepsilon\downarrow0$, we obtain $M\in\ov\Gamma_k$.
If $\sigma_k(M)>0$, the Newton--Maclaurin inequalities imply $M\in\Gamma_k$.
For small $\varepsilon>0$, the nearby contacts from below then satisfy
\[
M-\varepsilon I\in\Gamma_k,\qquad
0<\sigma_k(M-\varepsilon I)
\leq\sigma_k(D^2U^l)=f^l\longrightarrow0,
\]
a contradiction. Hence
\[
D^2U\in\ov\Gamma_k,\qquad \sigma_k(D^2U)=0
\quad\text{a.e. in }\R_+^n.
\]
The classification will use differentiated equations for the smooth
approximations. We therefore retain these approximations and the
normalized estimate \eqref{eq:normalized-source} in the model class below.

For fixed positive bounds $\kappa,K_0$, let $\cA(B)$ be the class of functions
$U\in C^{1,1}_{\mathrm{loc}}(\ov{\R_+^n})$ satisfying
\begin{equation}\label{eq:model-class}
\begin{gathered}
\kappa I_{n-1}\leq B\leq K_0 I_{n-1},\\
U(y,0)=\frac12\sum_{i,j=1}^{n-1}B_{ij}y^iy^j,\\
U(0)=0,\qquad DU(0)=0,\qquad
\norm{D^2U}_{L^\infty(\R_+^n)}\leq K_0,\\
D^2U\in\ov\Gamma_k,\quad \sigma_k(D^2U)=0\quad\text{a.e. in }\R_+^n,
\end{gathered}
\end{equation}
and the following approximation property. On a neighborhood of every
compact subset of $\R_+^n$, there are smooth functions $U^l\to U$
in $C^1$ such that
\begin{equation}\label{eq:approximation}
\begin{gathered}
\norm{D^2U^l}\leq K_0,\qquad D^2U^l\in\Gamma_k,\\
f^l=\sigma_k(D^2U^l)>0,\\
\norm{f^l}_{L^\infty}+
\norm{\frac{|Df^l|}{\sum_{i=1}^nF^{ii}(D^2U^l)}}_{L^\infty}\longrightarrow0.
\end{gathered}
\end{equation}

No boundary condition is imposed on the approximating functions.
Sequences in this class will always have common bounds $\kappa,K_0$.
By \eqref{eq:blowup-source} and \eqref{eq:normalized-source},
the limit constructed above belongs to $\cA(B)$.
The next theorem first classifies the functions in $\cA(B)$ and then proves Hessian convergence for solutions of the rescaled eigenvalue equation.

\begin{theorem}\label{thm:limits}
Let $2\leq k<n$.
\begin{enumerate}
\renewcommand{\labelenumi}{\textup{(\roman{enumi})}}
\item If $U\in\cA(B)$, then $U=Q_{B,b}$ for some $b\in\R^{n-1}$. Moreover,
$\bigl(F^{ij}(D^2Q_{B,b})\bigr)>0$.

\item Let $Q$ be a quadratic polynomial with Hessian $\bar M$ satisfying
\begin{equation}\label{eq:comparison-M}
\bar M\in\Gamma_{k-1},\qquad \sigma_k(\bar M)=0,\qquad
(F^{ij}(\bar M))>0.
\end{equation}
Set $D=B_{1/4}(e_n)$ and write $X=(y,t)$. Fix $\lambda>0$ and
$0<a_*\leq a^*<\infty$,
and let $r^l\downarrow0$, $a^l\in[a_*,a^*]$.
Suppose that $U^l\in C^2(D)$ are $k$-admissible solutions of
\begin{equation}\label{eq:rescaled-equation}
\sigma_k(D^2U^l)=\lambda (r^l)^k(a^l t-r^l U^l)^k
\quad\text{in }D,
\end{equation}
and $U^l\to Q$ uniformly in $D$. Then $D^2U^l\to\bar M$ uniformly
on compact subsets of $B_{1/8}(e_n)$.
\end{enumerate}
\end{theorem}

We prove \textup{(i)} and \textup{(ii)} in
Sections~\ref{sec:classification} and~\ref{sec:savin}, respectively.
We then return to the chosen sequence $x^l$ to reach a contradiction.
Figure~\ref{fig:gap-route} summarizes the argument.

\begin{figure}[htbp]
\centering
\begin{tikzpicture}[font=\footnotesize,>=Stealth]
\node[proof box] (a) at (0,2.1)
  {$x^l\to\partial\Omega$\\[3pt]$\lambda_{\min}(F^{ij}(D^2u(x^l)))\to0$};
\node[proof box] (b) at (5.1,2.1)
  {Boundary dilation $U^l$\\[3pt]$x^l\mapsto e_n$};
\node[proof box] (c) at (10.2,2.1)
  {$U\in\cA(B)$\\[3pt]$B>0$};
\node[proof box] (d) at (10.2,-.2)
  {$U=Q_{B,b}$\\[3pt]$(F^{ij}(\bar M_b))>0$};
\node[proof box] (e) at (5.1,-.2)
  {$D^2U^l(e_n)\to\bar M_b$};
\node[proof box,draw=figorange!65,fill=figorange!3] (f) at (0,-.2)
  {Positive limit\\[3pt]Contradiction};
\draw[proof arrow] (a)--node[above] {$(r^l)^{-1}$} (b);
\draw[proof arrow] (b)--node[above] {$C^1_{\mathrm{loc}}$} (c);
\draw[proof arrow] (c)--node[right,align=left]
  {Theorem\\\ref{thm:limits}\textup{(i)}} (d);
\draw[proof arrow] (d)--node[above,align=center,font=\scriptsize] {Theorem\\\ref{thm:limits}\textup{(ii)}} (e);
\draw[proof arrow] (e)--node[above] {continuity} (f);
\end{tikzpicture}
\caption{The proof of \eqref{eq:gap}. Compactness up to the boundary first
produces a $C^1$ limit. Theorem~\ref{thm:limits}\textup{(i)} identifies this limit,
and Theorem~\ref{thm:limits}\textup{(ii)} establishes the Hessian convergence
needed for the contradiction.}
\label{fig:gap-route}
\end{figure}
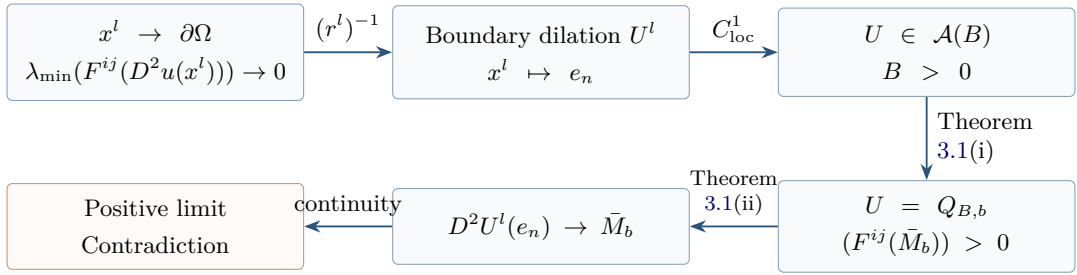

\subsection{Classification of the half-space limits}\label{sec:classification}

The classification below uses the same combination of derivative
quotients, the maximum principle, and compactness under rescaling
as Savin's classification of global half-space solutions of a degenerate
Monge--Amp\`ere equation
\cite[Section~7, Proposition~7.1]{SavinBoundary}.
Here the quotient uses tangential derivatives and the quadratic boundary
data. Its Newton coefficients may degenerate even in the interior,
so we also need the normalized coercivity estimate and boundary
inequality proved below.

Let $U\in\cA(B)$. The quadratic solutions in
Lemma~\ref{lem:quadratics} satisfy $D_yQ_{B,b}=By+tb$.
This suggests defining
\begin{equation}\label{eq:quotient}
p(y,t)=\frac{D_yU(y,t)-By}{t}.
\end{equation}
The boundary data and the Hessian bound imply
\[
D_yU(y,0)=By,\qquad
|D_yU(y,t)-D_yU(y,0)|\leq K_0t,
\]
and hence
\[
|p(y,t)|\leq K_0.
\]
We write $p=(p^1,\ldots,p^{n-1})$ and $X=(y,t)$, with $X^n=t$.
The quotient is locally Lipschitz in $\R_+^n$; no boundary trace
is assumed.

Suppose $p$ is constant, say $p\equiv b$. Integration in $y$ yields
\[
U(y,t)=\frac12\sum_{i,j=1}^{n-1}B_{ij}y^iy^j
+t\sum_{i=1}^{n-1}b^iy^i+v(t),\qquad v(t)=U(0,t).
\]
By \eqref{eq:block}, the equation becomes
\[
\sigma_k(B)+v''(t)\sigma_{k-1}(B)
-\sum_{i,j=1}^{n-1}[T_{k-2}(B)]^{ij}b^ib^j=0
\quad\text{a.e.}
\]
Thus $v''=c_b$. Since $v(0)=v'(0)=0$, we obtain
$v(t)=c_bt^2/2$ and $U=Q_{B,b}$.
To prove Theorem~\ref{thm:limits}\textup{(i)}, it therefore remains to show
that $p$ is a constant vector.

We first consider interior maxima, then maxima on the flat boundary,
and finally maximizing sequences tending to infinity. Dilation reduces
the last case to the first two.

\medskip
\noindent\emph{Interior maxima.}
Fix $\beta\in\R^{n-1}$ and set $W=|p-\beta|^2$.
We first work with a smooth approximating function, omit the superscript
$l$, and write $f=\sigma_k(D^2U)$.
Differentiating the identity
\[
\partial_a U-\sum_{b=1}^{n-1}B_{ab}y^b=tp^a
\qquad(1\leq a\leq n-1)
\]
twice and using
$\sum_{i,j=1}^nF^{ij}\partial_{ij}(\partial_a U)=\partial_af$,
we obtain
\begin{equation}\label{eq:quotient-divergence}
t\sum_{i,j=1}^nF^{ij}(D^2U)\partial_{ij}p^a
+2\sum_{i=1}^nF^{in}(D^2U)\partial_ip^a
=\partial_af.
\end{equation}
It follows that
\begin{equation}\label{eq:square-equation}
\begin{aligned}
&\sum_{i,j=1}^nF^{ij}(D^2U)\partial_{ij}W
+\frac2t\sum_{i=1}^nF^{in}(D^2U)\partial_iW\\
&\quad=2\sum_{a=1}^{n-1}\sum_{i,j=1}^n
F^{ij}(D^2U)\partial_ip^a\partial_jp^a
+\frac2t\sum_{a=1}^{n-1}(p^a-\beta^a)\partial_af.
\end{aligned}
\end{equation}
By \eqref{eq:approximation}, the normalized source term tends to zero
on interior compact sets. We must still handle the possible degeneracy
of $(F^{ij})$.

De Silva and Savin's higher order boundary Harnack theorem
\cite{DeSilvaSavin} establishes boundary regularity for quotients of
harmonic functions, and more generally of solutions to the same linear
elliptic equation, with a positive denominator in the domain and zero
data on a sufficiently regular boundary portion.
Terracini, Tortone and Vita \cite[Section~3.2]{TTV} study the associated
weighted divergence equations, with uniform ellipticity assumed after
the scalar weight is removed. Here the Newton tensor itself may
degenerate; Proposition~\ref{prop:coercivity} below uses its algebraic
structure to control this additional difficulty.

To compute the derivatives of $p$, write
\[
M=D^2U=\begin{pmatrix}C&d\\d^T&c\end{pmatrix},
\qquad V=[\,C-B,\ d-p\,].
\]
Differentiating \eqref{eq:quotient}, we obtain
\begin{equation}\label{eq:Dp}
t\partial_i p^a=V_{ai}
\qquad(1\leq a\leq n-1,\ 1\leq i\leq n).
\end{equation}
The next proposition uses $B>0$ and the form of $V$ to control it
even when $(F^{ij})$ is degenerate.

\begin{proposition}\label{prop:coercivity}
Fix $K_0,\kappa>0$. Suppose that
\[
\begin{gathered}
M=\begin{pmatrix}C&d\\d^T&c\end{pmatrix}\in\ov\Gamma_k,
\qquad\norm{M}\leq K_0,\\
\kappa I_{n-1}\leq B\leq K_0 I_{n-1},\qquad |p|\leq K_0,
\end{gathered}
\]
and set $V=[\,C-B,\ d-p\,]$. If $2\leq k<n$, there is
$\gamma>0$, depending only on these bounds, $n$, and $k$, such that
\begin{equation}\label{eq:coercivity}
\sum_{a=1}^{n-1}\sum_{i,j=1}^n F^{ij}(M)V_{ai}V_{aj}
\geq\gamma\left(\sum_{i=1}^nF^{ii}(M)\right)\norm{V}^2.
\end{equation}
If $\bar F^{ij}$ is any normalized limit associated with $M$ as in
Lemma~\ref{lem:newton-limits}\textup{(ii)}, then
\begin{equation}\label{eq:normalized-coercivity}
\sum_{a=1}^{n-1}\sum_{i,j=1}^n\bar F^{ij}V_{ai}V_{aj}
\geq\gamma\norm{V}^2.
\end{equation}
\end{proposition}

\begin{proof}
When $\sum_iF^{ii}(M)=0$, the estimate is immediate.
If \eqref{eq:coercivity} fails, a sequence with positive traces and
$V^l\ne0$ satisfies
\[
\frac{\sum_{a=1}^{n-1}\sum_{i,j=1}^nF^{ij}(M^l)V^l_{ai}V^l_{aj}}
{\bigl(\sum_{r=1}^nF^{rr}(M^l)\bigr)\norm{V^l}^2}
\longrightarrow0.
\]
After passing to a subsequence,
\[
\begin{gathered}
M^l\to M,\qquad B^l\to B,\qquad p^l\to p,\\
V^l\to V=[\,C-B,\ d-p\,],\qquad
\frac{F^{ij}(M^l)}{\sum_{r=1}^nF^{rr}(M^l)}\to\bar F^{ij}.
\end{gathered}
\]
If $V=0$, then $C=B>0$. Interlacing and
Lemma~\ref{lem:newton-limits}\textup{(i)} imply $(F^{ij}(M))>0$.
For some $\theta>0$ and all large $l$,
\[
\left(\frac{F^{ij}(M^l)}{\sum_{r=1}^nF^{rr}(M^l)}\right)
\geq\theta I,
\qquad
\frac{\sum_{a=1}^{n-1}\sum_{i,j=1}^nF^{ij}(M^l)V^l_{ai}V^l_{aj}}
{\bigl(\sum_{r=1}^nF^{rr}(M^l)\bigr)\norm{V^l}^2}
\geq\theta,
\]
a contradiction.

If $V\ne0$, then
\[
(\bar F^{ij})\geq0,\qquad
\sum_{a=1}^{n-1}\sum_{i,j=1}^n\bar F^{ij}V_{ai}V_{aj}=0.
\]
The quadratic form vanishes on each row, so
\[
\sum_{h=1}^nV_{ah}\bar F^{hj}=0
\qquad(1\leq a<n,\ 1\leq j\leq n).
\]
Since $V\ne0$, $(\bar F^{ij})$ is degenerate. By
Lemma~\ref{lem:newton-limits}\textup{(ii)},
\[
\sum_{h=1}^nM_{ih}\bar F^{hj}=0,\qquad
\rank(\bar F^{ij})\geq n-k+1\geq2.
\]
Its range therefore contains a nonzero vector $\xi=(\xi',0)$, and
\[
M\xi=0,\qquad V\xi=0,\qquad C\xi'=0,
\qquad 0=V\xi=(C-B)\xi'=-B\xi'\ne0.
\]
This proves \eqref{eq:coercivity}.

To prove \eqref{eq:normalized-coercivity}, take a sequence
$M^l\to M$ defining the normalized limit and write
\[
M^l=\begin{pmatrix}C^l&d^l\\(d^l)^T&c^l\end{pmatrix},\qquad
V^l=[\,C^l-B,\ d^l-p\,]\longrightarrow V.
\]
For large $l$, $\norm{M^l}\leq2K_0$. By \eqref{eq:coercivity}
with this bound,
\[
\begin{aligned}
\sum_{a=1}^{n-1}\sum_{i,j=1}^n\bar F^{ij}V_{ai}V_{aj}
&=\lim_{l\to\infty}\sum_{a=1}^{n-1}\sum_{i,j=1}^n
\frac{F^{ij}(M^l)V^l_{ai}V^l_{aj}}
{\sum_{r=1}^nF^{rr}(M^l)}\\
&\geq\gamma\lim_{l\to\infty}\norm{V^l}^2
=\gamma\norm V^2.
\end{aligned}
\]
After decreasing $\gamma$ if necessary, both estimates hold with the same constant.
\end{proof}

Under the stated bounds, combining \eqref{eq:Dp} with the proposition yields
\[
\sum_{a=1}^{n-1}\sum_{i,j=1}^n
F^{ij}\partial_ip^a\partial_jp^a
\geq\gamma\left(\sum_{i=1}^nF^{ii}\right)
\sum_{a=1}^{n-1}\sum_{i=1}^n|\partial_ip^a|^2
\]
To pass this estimate to the $C^{1,1}$ limit, we use functions touching
$W$ from above.
For a symmetric matrix $L=(L^{ij})\geq0$, write
\begin{equation}\label{eq:contact-operator}
\cL_L\phi=\sum_{i,j=1}^nL^{ij}\partial_{ij}\phi
+\frac2t\sum_{i=1}^nL^{in}\partial_i\phi.
\end{equation}

The next lemma combines \eqref{eq:square-equation} with the coercivity
estimate and passes the result to upper test functions. When the
derivatives of the test function are small, it also shows that the
normalized Newton tensor at the contact point is uniformly elliptic.

\begin{lemma}\label{lem:contact}
Let $U\in\cA(B)$, $\beta\in\R^{n-1}$, and $W=|p-\beta|^2$.
Suppose that $\phi\in C^2$ touches $W$ from above at
$X=(y,t)\in\R_+^n$. Then there are
\[
M=\begin{pmatrix}C&d\\ d^T&c\end{pmatrix}\in\ov\Gamma_k,
\qquad \norm{M}\leq K_0,\qquad \sigma_k(M)=0,
\]
and a normalized limit $(\bar F^{ij})$ associated with $M$
as in Lemma~\ref{lem:newton-limits}\textup{(ii)}, such that,
with $V=[\,C-B,\ d-p(X)\,]$,
\begin{equation}\label{eq:contact-data}
\begin{gathered}
(\bar F^{ij})\geq0,\qquad
\sum_{i=1}^n\bar F^{ii}=1,\\
\cL_{\bar F}\phi(X)\geq\frac{2\gamma}{t^2}\norm{V}^2.
\end{gathered}
\end{equation}
Moreover, there are $\delta,\theta_0>0$, depending only on
$n,k,K_0,\kappa$, such that
\begin{equation}\label{eq:contact-smallness}
t^2\norm{D^2\phi(X)}+t|D\phi(X)|\leq\delta
\end{equation}
implies
\begin{equation}\label{eq:contact-ellipticity}
\theta_0I\leq(\bar F^{ij})\leq I.
\end{equation}
\end{lemma}

\begin{proof}
We use the standard stability argument for upper contacts under uniform
convergence \cite[Lemma~6.1, p.~33]{CIL}, keeping track of the Hessians
and normalized Newton tensors of the approximating solutions.
After replacing $\phi(Z)$ by $\phi(Z)+|Z-X|^4$, the contact is strict.
On a small closed ball in $\R_+^n$, the smooth approximations satisfy
\[
p^l=\frac{D_yU^l-By}{t}\longrightarrow p,
\qquad W^l:=|p^l-\beta|^2\longrightarrow W
\quad\text{uniformly}.
\]
Choose interior maxima $X^l=(y^l,t^l)$ of $W^l-\phi$. Then
\[
X^l\to X,\qquad DW^l(X^l)=D\phi(X^l),\qquad
D^2W^l(X^l)\leq D^2\phi(X^l).
\]
By \eqref{eq:approximation},
\[
\frac{2|p^l(X^l)-\beta|}{t^l}
\frac{|Df^l(X^l)|}{\sum_{i=1}^nF^{ii}(D^2U^l(X^l))}
\longrightarrow0.
\]
After passing to a subsequence,
\[
\begin{gathered}
D^2U^l(X^l)\to M\in\ov\Gamma_k,\qquad
\norm M\leq K_0,\qquad \sigma_k(M)=0,\\
\frac{F^{ij}(D^2U^l(X^l))}{\sum_{r=1}^nF^{rr}(D^2U^l(X^l))}
\to\bar F^{ij},\qquad
V^l\to V=[\,C-B,\ d-p(X)\,].
\end{gathered}
\]
Divide \eqref{eq:square-equation} by the trace and pass to the limit.
By Proposition~\ref{prop:coercivity},
\[
\cL_{\bar F}\phi(X)
\geq\frac2{t^2}\sum_{a=1}^{n-1}\sum_{i,j=1}^n
\bar F^{ij}V_{ai}V_{aj}
\geq\frac{2\gamma}{t^2}\norm V^2.
\]
This proves \eqref{eq:contact-data}.

Since $(\bar F^{ij})\geq0$ and $\sum_i\bar F^{ii}=1$,
\begin{equation}\label{eq:V-small}
2\gamma\norm V^2
\leq t^2\norm{D^2\phi(X)}+2t|D\phi(X)|.
\end{equation}
Choose $\delta\leq\gamma\kappa^2/4$. Under
\eqref{eq:contact-smallness},
\[
\norm V^2\leq\frac\delta\gamma\leq\frac{\kappa^2}{4},\qquad
C\geq B-\norm{C-B}I\geq\frac\kappa2I.
\]
Lemma~\ref{lem:quadratics}, applied to $C$, and continuity imply
\[
\bar F^{ij}=\frac{F^{ij}(M)}{\sum_{r=1}^nF^{rr}(M)},\qquad
\theta_0I\leq(\bar F^{ij})\leq I,
\]
where $\theta_0>0$ depends only on the common bounds.
This proves \eqref{eq:contact-ellipticity}.
\end{proof}

We now prove the interior maximum principle using a barrier that
satisfies \eqref{eq:contact-smallness}.

\begin{lemma}\label{lem:interior-maximum}
Let $U\in\cA(B)$ and $\beta\in\R^{n-1}$.
If $|p-\beta|^2$ attains its global maximum at a point of $\R_+^n$,
then $p$ is constant.
\end{lemma}

\begin{proof}
Set
\[
W=|p-\beta|^2,\qquad m=\sup_{\R_+^n}W,\qquad S=\{W=m\}.
\]
If $S$ is not open, choose a relative boundary point
$X_0=(y_0,t_0)\in S$ and $z\in B_{t_0/4}(X_0)\setminus S$.
With $R=\dist(z,S)$,
\[
\begin{gathered}
0<R\leq|z-X_0|<t_0/4,\qquad
z^n-2R\geq t_0-3|z-X_0|>t_0/4,\\
B_{2R}(z)\subset\R_+^n,\qquad
\max_{\partial B_{R/2}(z)}W<m,\qquad
S\cap\partial B_R(z)\ne\varnothing.
\end{gathered}
\]
Consider
\begin{equation}\label{eq:radial-barrier}
\phi(X)=m-\varepsilon
\bigl(e^{-a|X-z|^2}-e^{-aR^2}\bigr).
\end{equation}
Then
\[
\begin{aligned}
\partial_i\phi
&=2\varepsilon a e^{-a|X-z|^2}(X^i-z^i),\\
\partial_{ij}\phi
&=2\varepsilon a e^{-a|X-z|^2}
\bigl(\delta_{ij}-2a(X^i-z^i)(X^j-z^j)\bigr),
\end{aligned}
\]
and
\begin{equation}\label{eq:barrier-calculation}
\begin{aligned}
\frac{\cL_L\phi}{2\varepsilon a e^{-a|X-z|^2}}
&=\sum_{i=1}^nL^{ii}
+\frac2t\sum_{i=1}^nL^{in}(X^i-z^i)\\
&\quad-2a\sum_{i,j=1}^nL^{ij}(X^i-z^i)(X^j-z^j).
\end{aligned}
\end{equation}
For $\theta_0I\leq L\leq I$ and $R/2\leq|X-z|\leq R$,
choose $a$ sufficiently large so that
\[
\frac{\cL_L\phi}{2\varepsilon a e^{-a|X-z|^2}}
\leq n+\frac{2R}{z^n-R}-\frac{a\theta_0R^2}{2}<0
\]
Fix this $a$ and choose
$\varepsilon>0$ so that
\[
\begin{gathered}
\varepsilon\bigl(e^{-aR^2/4}-e^{-aR^2}\bigr)
\leq m-\max_{\partial B_{R/2}(z)}W,\\
\sup_{\ov B_R(z)\setminus B_{R/2}(z)}
\bigl(t^2\norm{D^2\phi}+t|D\phi|\bigr)\leq\delta.
\end{gathered}
\]
Thus $W\leq\phi$ on both boundary spheres. At a positive interior
maximum of $W-\phi$, Lemma~\ref{lem:contact} would imply
\[
\theta_0I\leq(\bar F^{ij})\leq I,\qquad
0\leq\cL_{\bar F}\phi<0.
\]
a contradiction.
Consequently,
\[
W\leq\phi\quad\text{in }\ov B_R(z)\setminus B_{R/2}(z).
\]
Choose $X_*\in S\cap\partial B_R(z)$. Since $\phi\geq m\geq W$
outside $B_R(z)$,
\[
W(X_*)=\phi(X_*)=m,\qquad
W\leq\phi\quad\text{in }B_{R/4}(X_*).
\]
Lemma~\ref{lem:contact} again yields
$0\leq\cL_{\bar F}\phi(X_*)<0$.
Hence $S$ is open as well as relatively closed and nonempty, so
$W\equiv m$.

Write $m=r^2$, fix $b=p(X_0)$ at an interior point, and set
$\beta'=2\beta-b$. 
The values of $p$ lie on the sphere of radius $r$ centered at $\beta$.
The segment joining $\beta'$ and $b=p(X_0)$ is a diameter of this sphere. Hence
\begin{equation}\label{eq:reflected-center}
|p-\beta'|^2=4r^2-|p-b|^2\leq4r^2,
\end{equation}
with equality at $X_0$. Applying the same strong maximum principle
with $\beta'$ in place of $\beta$, we obtain
\[
|p-\beta'|^2\equiv4r^2.
\]
By \eqref{eq:reflected-center}, $|p-b|^2\equiv0$,
and hence $p\equiv b$. 
\end{proof}

\medskip
\noindent\emph{Maxima on the flat boundary.}
Since no boundary trace of the quotient is assumed, we use an upper
envelope on the closed half-space. We formulate the argument for
sequences so that it also applies to the dilations used below. Let $U^l\in\cA(B^l)$ have common bounds, with $B^l\to B$, and set
$p^l=(D_yU^l-B^ly)/t$. For a fixed $\beta\in\R^{n-1}$,
define on the closed half-space
\begin{equation}\label{eq:envelope}
\mathcal W(X)=\limsup_{\substack{l\to\infty,\ Z\to X\\Z\in\R_+^n}}
|p^l(Z)-\beta|^2.
\end{equation}
The function $\mathcal W$ is bounded and upper semicontinuous.
For compact $K\subset\ov{\R_+^n}$ and $\psi\in C(K)$, choose points
in $K\cap\R_+^n$ approaching the suprema of $|p^l-\beta|^2-\psi$.
Compactness yields a subsequence $X^l\to X\in K$, and the definition
of the envelope implies
\[
\limsup_{l\to\infty}
\bigl(|p^l(X^l)-\beta|^2-\psi(X^l)\bigr)
\leq\mathcal W(X)-\psi(X).
\]
Consequently,
\begin{equation}\label{eq:compact-envelope}
\limsup_{l\to\infty}\sup_{K\cap\R_+^n}
\bigl(|p^l-\beta|^2-\psi\bigr)
\leq\sup_K(\mathcal W-\psi)
\end{equation}

The following boundary inequality and the interior maximum principle
together handle every maximum attained at a finite point.

\begin{lemma}\label{lem:boundary-test}\label{lem:finite-maximum}
Let $U^l\in\cA(B^l)$ have common bounds, with $B^l\to B$.
Fix $\beta\in\R^{n-1}$ and define $\mathcal W$ by
\eqref{eq:envelope}. There is $C_0>0$, depending only on
$n,k,K_0,\kappa$, such that
every $C^2$ upper test $\phi$ for $\mathcal W$ at
$X_0=(y_0,0)$ with $\mathcal W(X_0)>0$ satisfies
\begin{equation}\label{eq:weak-oblique}
\partial_t\phi(X_0)+C_0|D_y\phi(X_0)|\geq0.
\end{equation}
Moreover, if $U^l\to U\in\cA(B)$ locally in
$C^1(\ov{\R_+^n})$ and $\mathcal W$ attains its global maximum
at a finite point of $\ov{\R_+^n}$, then the quotient $p$ of $U$
is constant.
\end{lemma}

\begin{proof}
Write $W^l=|p^l-\beta|^2$. At a boundary upper contact $X_0$,
choose $r>0$ and set
\[
\psi=\phi+|X-X_0|^4,\qquad
\mathcal W-\psi\leq-|X-X_0|^4
\quad\text{on }\ov B_r(X_0)\cap\ov{\R_+^n}.
\]
After passing to a subsequence and relabeling, choose
\[
X^l=(y^l,t^l)\to X_0,\qquad t^l>0,\qquad
W^l(X^l)\to\mathcal W(X_0),\qquad \varepsilon^l=t^l/l.
\]
Maximize
\begin{equation}\label{eq:pole-test}
W^l(X)-\psi(X)-\frac{\varepsilon^l}{t}
\end{equation}
on $\ov B_r(X_0)\cap\R_+^n$. Its value at $X^l$ tends to zero,
whereas, by \eqref{eq:compact-envelope},
\[
\sup_{\partial B_r(X_0)\cap\R_+^n}(W^l-\psi)\leq-r^4/2
\quad\text{for large }l.
\]
For fixed $l$, the expression tends to $-\infty$ at $t=0$.
Hence the maximum occurs at an interior point $Y^l$.
Writing $s^l=(Y^l)^n$, we have
\[
W^l(X^l)-\psi(X^l)-\frac1l
\leq W^l(Y^l)-\psi(Y^l)-\frac{\varepsilon^l}{s^l}
\leq W^l(Y^l)-\psi(Y^l).
\]
The strict contact and \eqref{eq:compact-envelope} therefore imply
\begin{equation}\label{eq:pole-limits}
Y^l\to X_0,\qquad W^l(Y^l)\to\mathcal W(X_0),\qquad
\frac{\varepsilon^l}{s^l}\to0.
\end{equation}

For each fixed $l$, apply Lemma~\ref{lem:contact} at $Y^l$ with
test $\psi+\varepsilon^l/t$, up to an additive constant.
Here we first pass to the limit in the smooth approximations, and only
then let $l\to\infty$.
Since
\begin{equation}\label{eq:pole-sign}
\cL_{\bar F}(\varepsilon/t)
=\frac{2\varepsilon}{t^3}\bar F^{nn}
-\frac{2\varepsilon}{t^3}\bar F^{nn}=0,
\end{equation}
the matrices at the contact points satisfy
\begin{equation}\label{eq:boundary-contact}
\sum_{i,j=1}^n\bar F^{ij,l}\partial_{ij}\psi(Y^l)
+\frac2{s^l}\sum_{i=1}^n\bar F^{in,l}\partial_i\psi(Y^l)
\geq\frac{2\gamma}{(s^l)^2}\norm{V^l}^2.
\end{equation}
Thus
\[
2\gamma\norm{V^l}^2
\leq(s^l)^2\norm{D^2\psi(Y^l)}+2s^l|D\psi(Y^l)|\longrightarrow0.
\]
After passing to a subsequence, $p^l(Y^l)\to b$, $|b|\leq K_0$. Since
$V^l=[\,C^l-B^l,\ d^l-p^l(Y^l)\,]$, it follows from \eqref{eq:block} that
\[
\begin{gathered}
C^l\to B,\qquad d^l\to b,\\
c^l=\frac{\sum_{i,j=1}^{n-1}[T_{k-2}(C^l)]^{ij}
(d^l)^i(d^l)^j-\sigma_k(C^l)}{\sigma_{k-1}(C^l)}
\longrightarrow c_b,\qquad M^l\to\bar M_b.
\end{gathered}
\]
Lemma~\ref{lem:quadratics} and continuity imply, for large $l$,
\[
\bar F^{ij,l}=\frac{F^{ij}(M^l)}{\sum_rF^{rr}(M^l)}
\longrightarrow\frac{F^{ij}(\bar M_b)}{\sum_rF^{rr}(\bar M_b)}.
\]
Multiplying \eqref{eq:boundary-contact} by $s^l/2$ and passing to the limit, we obtain
\[
\frac{\sum_iF^{in}(\bar M_b)\partial_i\phi(X_0)}
{\sum_rF^{rr}(\bar M_b)}\geq0.
\]
By \eqref{eq:normal-column},
\begin{equation}\label{eq:oblique}
\sigma_{k-1}(B)\partial_t\phi(X_0)
-\sum_{a,i=1}^{n-1}[T_{k-2}(B)]^{ai}b^i\partial_a\phi(X_0)\geq0.
\end{equation}
Moreover,
\[
\frac{\left(\sum_{a=1}^{n-1}
\left|\sum_{i=1}^{n-1}[T_{k-2}(B)]^{ai}b^i\right|^2\right)^{1/2}}
{\sigma_{k-1}(B)}
\leq\frac{k-1}{n-1}\left(\frac{K_0}{\kappa}\right)^{k-1}=:C_0.
\]
Dividing \eqref{eq:oblique} by $\sigma_{k-1}(B)$ proves
\eqref{eq:weak-oblique}.

For the last assertion, set $W=|p-\beta|^2$ and $m=\max\mathcal W$.
Local $C^1$ convergence implies $\mathcal W=W$ for $t>0$.
If $m=0$, then $p=\beta$; if $m$ is attained in $\R_+^n$,
Lemma~\ref{lem:interior-maximum} applies. Otherwise,
\[
\mathcal W(X_0)=m>0,\qquad X_0=(y_0,0),\qquad
W<m\quad\text{in }\R_+^n.
\]
Fix $R>0$, set $z=(y_0,R)$, and consider
\[
\phi(X)=m-\varepsilon
\left(e^{-a|X-z|^2}-e^{-aR^2}\right),
\]
where $a$ is chosen sufficiently large and then
$\varepsilon>0$ sufficiently small.
For $X\in\ov B_R(z)\cap\R_+^n$,
\[
|y-y_0|^2+(t-R)^2\leq R^2,
\qquad |y-y_0|^2\leq2Rt.
\]
Hence, for $\theta_0I\leq L\leq I$,
\begin{equation}\label{eq:drift-estimate}
\begin{aligned}
\frac2t\sum_iL^{in}(X^i-z^i)
&\leq2-\frac{2\theta_0R}{t}+2\sqrt{\frac{2R}{t}}\\
&\leq-\frac{\theta_0R}{t}+2+\frac2{\theta_0}.
\end{aligned}
\end{equation}
On $R/2\leq|X-z|\leq R$, \eqref{eq:barrier-calculation} implies, for large $a$,
\[
\frac{\cL_L\phi}{2\varepsilon a e^{-a|X-z|^2}}
\leq n+2+\frac2{\theta_0}-\frac{a\theta_0R^2}{2}<0
\]
Fix this $a$ and choose $\varepsilon$ small enough that
\[
\begin{gathered}
W\leq\phi\quad\text{on }\partial B_{R/2}(z),\qquad
\mathcal W\leq m=\phi\quad\text{on }\partial B_R(z),\\
\sup_{\ov B_R(z)\setminus B_{R/2}(z)}
\bigl(t^2\norm{D^2\phi}+t|D\phi|\bigr)\leq\delta.
\end{gathered}
\]
Any positive maximum of $\mathcal W-\phi$ would occur at positive height
and contradict Lemma~\ref{lem:contact}. Thus
\[
\mathcal W\leq\phi\quad\text{on the closed annulus},\qquad
\phi\geq m\geq\mathcal W\quad\text{outside }B_R(z).
\]
The barrier is therefore an upper test throughout a neighborhood
of $X_0$ in the closed half-space, where
\[
D_y\phi(X_0)=0,\qquad
\partial_t\phi(X_0)=-2\varepsilon aR e^{-aR^2}<0.
\]
This contradicts \eqref{eq:weak-oblique}.
\end{proof}

\medskip
\noindent\emph{Maximizing sequences at infinity.}
Dilation brings a maximizing sequence to a bounded set. We first
establish compactness and invariance under dilation. For $U\in\cA(B)$,
\[
|U(X)|\leq\frac{K_0}{2}|X|^2,\qquad
|DU(X)|\leq K_0|X|,\qquad
\norm{D^2U}\leq K_0\quad\text{a.e.}
\]
Thus $U^l\in\cA(B^l)$, $B^l\to B$, have a subsequence with
\[
U^l\longrightarrow U\quad\text{in }C^1_{\mathrm{loc}}(\ov{\R_+^n}).
\]
The boundary data and Hessian bound pass to $U$. A diagonal choice
of smooth approximations establishes \eqref{eq:approximation};
the contact argument in Section~\ref{sec:boundary-dilation} proves
admissibility and the equation with zero right-hand side. Hence $U\in\cA(B)$.

For each fixed $R>0$, the dilation $U(X)\mapsto R^{-2}U(RX)$
preserves the boundary data, normalization, and Hessian bound.
For the smooth approximations,
\[
\sigma_k\bigl(D^2[R^{-2}U^l(RX)]\bigr)=f^l(RX),\qquad
\frac{|D_X[f^l(RX)]|}{\sum_iF^{ii}(D^2U^l(RX))}
=R\frac{|Df^l(RX)|}{\sum_iF^{ii}(D^2U^l(RX))}\longrightarrow0.
\]
Thus every dilation belongs to $\cA(B)$, and compactness applies
to any sequence of scales.

\begin{lemma}\label{lem:maximum-infinity}
Let $U\in\cA(B)$. If there are points $X^l\in\R_+^n$ such that
\[
|X^l|\longrightarrow\infty,\qquad
|p(X^l)|^2\longrightarrow\sup_{\R_+^n}|p|^2,
\]
then $p$ is constant.
\end{lemma}

\begin{proof}
Set
\[
m=\sup_{\R_+^n}|p|^2,\qquad R^l=|X^l|,\qquad
U^l(X)=\frac{U(R^lX)}{(R^l)^2}.
\]
After passing to a subsequence,
\[
\begin{gathered}
U^l\to U_\infty\in\cA(B)
\quad\text{in }C^1_{\mathrm{loc}}(\ov{\R_+^n}),\\
X^l/R^l\to X_*\in\partial B_1\cap\ov{\R_+^n},\qquad
p^l(X)=p(R^lX),\\
|p^l(X^l/R^l)|^2=|p(X^l)|^2\longrightarrow m.
\end{gathered}
\]
For the envelope with $\beta=0$,
\[
\mathcal W\leq m=\mathcal W(X_*).
\]
By Lemma~\ref{lem:finite-maximum}, the quotient of $U_\infty$
is a constant vector $b$.

To show that the original quotient is constant, we need convergence
that remains uniform as the height tends to zero.
Consider the envelope with center $\beta=b$. Interior $C^1$ convergence
implies $\mathcal W=0$ for $t>0$. If $\mathcal W(y_0,0)>0$, choose
$A>\sup\mathcal W$ and let
\[
c:=\max_{|y-y_0|\leq1}
\bigl(\mathcal W(y,0)-A|y-y_0|^2\bigr)
=\mathcal W(y_*,0)-A|y_*-y_0|^2>0.
\]
Then $|y_*-y_0|<1$. For $T>2C_0A|y_*-y_0|$, define
\[
\phi(y,t)=c+A|y-y_0|^2-Tt.
\]
On the flat boundary, $\phi\geq\mathcal W$ with equality at $(y_*,0)$;
at nearby points of positive height, $\phi>0=\mathcal W$.
Thus \eqref{eq:weak-oblique} implies
\[
0\leq\partial_t\phi(y_*,0)+C_0|D_y\phi(y_*,0)|
=-T+2C_0A|y_*-y_0|<0.
\]
Hence $\mathcal W=0$ on the closed half-space. Applying
\eqref{eq:compact-envelope}, we obtain, for every $R>0$,
\begin{equation}\label{eq:no-layer}
\sup_{\ov B_R\cap\R_+^n}|p^l-b|\longrightarrow0.
\end{equation}

For any fixed $Y\in\R_+^n$ and all large $l$,
\[
Y/R^l\in B_1\cap\R_+^n,\qquad
|p(Y)-b|=|p^l(Y/R^l)-b|
\leq\sup_{\ov B_1\cap\R_+^n}|p^l-b|\longrightarrow0.
\]
Therefore $p\equiv b$.
\end{proof}

\begin{proof}[Proof of Theorem~\ref{thm:limits}\textup{(i)}]
Let $U\in\cA(B)$ and define $p$ by \eqref{eq:quotient}.
Choose $X^l\in\R_+^n$ such that
\[
|p(X^l)|^2\longrightarrow m:=\sup_{\R_+^n}|p|^2.
\]
After passing to a subsequence, either $X^l\to X_0$ in the
closed half-space or $|X^l|\to\infty$.

If $X_0\in\R_+^n$, continuity implies $|p(X_0)|^2=m$, and
Lemma~\ref{lem:interior-maximum} implies that $p$ is constant.
If $X_0=(y_0,0)$, use \eqref{eq:envelope} for the constant sequence
$U^l=U$ with $\beta=0$. Then
$\mathcal W\leq m=\mathcal W(X_0)$, and
Lemma~\ref{lem:finite-maximum} implies the same conclusion.
If $|X^l|\to\infty$, apply Lemma~\ref{lem:maximum-infinity}.

Thus $p\equiv b$ for some $b\in\R^{n-1}$. The argument at the
beginning of this subsection shows that $U=Q_{B,b}$, and
Lemma~\ref{lem:quadratics} implies
$\bigl(F^{ij}(D^2Q_{B,b})\bigr)>0$.
\end{proof}

\subsection{Hessian convergence and uniform ellipticity}\label{sec:savin}

Theorem~\ref{thm:limits}\textup{(i)} identifies the limit as $Q_{B,b}$
and shows that the linearization at its Hessian is positive definite.
Savin's small perturbation theorem
\cite[Theorem~1.3, p.~560]{Savin} is the key tool for passing from
uniform convergence to Hessian convergence. We verify its hypotheses
below to obtain a uniform $C^{2,\beta}$ estimate before passing to
the second derivatives. We first write the equation in a form that is
monotone for all symmetric matrices and uniformly elliptic near the
limiting Hessian.

For $M\in\Sym(n)$, define
\[
\rho_0(M)=\sup\{s:M-sI\in\Gamma_k\}.
\]
The matrix $M-sI$ is positive definite for sufficiently negative $s$.
Monotonicity of $\Gamma_k$ and the condition $\sigma_1>0$ therefore imply
\[
\{s:M-sI\in\Gamma_k\}=(-\infty,\rho_0(M)),\qquad
\rho_0(M)\leq\sigma_1(M)/n.
\]
By continuity, the endpoint matrix lies in the closed cone. If
$\sigma_k$ were positive there, the Newton--Maclaurin inequalities
would place the matrix in $\Gamma_k$, contrary to the definition of
$\rho_0$. Hence
\[
M-\rho_0(M)I\in\ov\Gamma_k,\qquad
\sigma_k(M-\rho_0(M)I)=0.
\]
Moreover, for $s<\rho_0(M)$,
\[
\begin{gathered}
\frac{d}{ds}\sigma_k(M-sI)=-\sum_{i=1}^nF^{ii}(M-sI)<0,\\
\lim_{s\to-\infty}\sigma_k(M-sI)=\infty,\qquad
\lim_{s\uparrow\rho_0(M)}\sigma_k(M-sI)=0.
\end{gathered}
\]
Thus, for every $f>0$, there is a unique $\rho_f(M)$ such that
\begin{equation}\label{eq:rho-definition}
M-\rho_f(M)I\in\Gamma_k,\qquad
\sigma_k(M-\rho_f(M)I)=f.
\end{equation}
For $f\geq0$, matrix monotonicity and translation by multiples of the
identity, together with comparison of $M$ and $N$ after adding
$\norm{M-N}_{\mathrm{op}}I$, imply
\begin{equation}\label{eq:rho-properties}
\begin{gathered}
\rho_f(M+P)\geq\rho_f(M)\qquad(P\geq0),\qquad
\rho_f(M+sI)=\rho_f(M)+s,\\
|\rho_f(M)-\rho_f(N)|\leq\norm{M-N}_{\mathrm{op}}.
\end{gathered}
\end{equation}
Also,
\[
\rho_0(M)=0\quad\Longleftrightarrow\quad
M\in\ov\Gamma_k,\quad \sigma_k(M)=0.
\]
This construction by shifts in the identity direction is a special case
of the canonical operators in \cite[Sections~11.4--11.5]{CHLP}.
For $f>0$, implicit differentiation of \eqref{eq:rho-definition} yields
\begin{equation}\label{eq:rho-derivative}
\begin{aligned}
\frac{\partial\rho_f(M)}{\partial M_{ij}}
&=\frac{F^{ij}(M-\rho_f(M)I)}
{\sum_hF^{hh}(M-\rho_f(M)I)},\\
\frac{\partial\rho_f(M)}{\partial f}
&=-\frac1{\sum_hF^{hh}(M-\rho_f(M)I)}.
\end{aligned}
\end{equation}
If $\bar M$ satisfies \eqref{eq:comparison-M}, the denominator is
positive at $(\bar M,0)$. The implicit function theorem therefore
extends $\rho_f(M)$ smoothly to a neighborhood of this point, including
small negative values of $f$. For $f\geq0$, positivity of the lower
symmetric functions shows that this extension agrees with the definition
above. After shrinking the neighborhood,
\[
cI\leq\left(\frac{\partial\rho_f}{\partial M_{ij}}\right)\leq CI.
\]
At a fixed positive level, the same normalization and matrix derivative
appear in \cite[(2.7)--(2.8), p.~4]{ZZ}. For convex solutions with
positive constant right-hand side, Zhou and Zhu also use Savin's
theorem to obtain Hessian convergence \cite[Lemma~4.5, pp.~13--14]{ZZ}.
Here the right-hand side depends on the point and the solution and
tends to zero. The comparison construction below verifies the
hypotheses of Savin's original theorem in this setting.

\begin{proof}[Proof of Theorem~\ref{thm:limits}\textup{(ii)}]
Fix $0<\alpha<1$. To meet the zero-solution hypothesis in Savin's
theorem, we first construct an exact solution close to $Q$ of the same
equation as $U^l$:
\begin{equation}\label{eq:comparison-equation}
\begin{cases}
\sigma_k(D^2v^l)=\lambda(r^l)^k(a^lt-r^lv^l)^k&\text{in }D,\\
v^l=Q&\text{on }\partial D,
\end{cases}
\end{equation}
with
\begin{equation}\label{eq:comparison-bound}
\norm{v^l-Q}_{C^{4,\alpha}(\ov D)}\leq C(r^l)^k.
\end{equation}
Subtracting $v^l$ then produces a transformed equation solved by
both zero and $U^l-v^l$.

Let $Lh=\sum_{i,j=1}^nF^{ij}(\bar M)\partial_{ij}h$. By
\cite[Theorem~6.14, p.~107; Theorem~6.6, pp.~98--99;
Theorem~6.19, pp.~111--112]{GT}, the zero Dirichlet inverse satisfies
\[
\norm{L^{-1}g}_{C^{4,\alpha}(\ov D)}
\leq C\norm g_{C^{2,\alpha}(\ov D)}
\]
Writing $v^l=Q+h$, with $h=0$ on $\partial D$, reduces
\eqref{eq:comparison-equation} to
\[
h=L^{-1}\left[
\lambda(r^l)^k(a^lt-r^l(Q+h))^k
-\sigma_k(\bar M+D^2h)+Lh\right].
\]
On the ball $\norm h_{C^{4,\alpha}}\leq C_0(r^l)^k$, Taylor's formula
implies
\[
\bigl\|\sigma_k(\bar M+D^2h_1)-\sigma_k(\bar M+D^2h_2)
-L(h_1-h_2)\bigr\|_{C^{2,\alpha}}\leq CC_0(r^l)^k\norm{h_1-h_2}_{C^{4,\alpha}}.
\]
The source term has Lipschitz constant at most $C(r^l)^{k+1}$ on the
same ball, and the map on the right takes $h=0$ to a function of norm at
most $C(r^l)^k$. Choose $C_0$ large enough and then $l$ large enough.
The map then preserves the ball and has Lipschitz constant at most
$1/2$. The contraction theorem proves
\eqref{eq:comparison-equation}--\eqref{eq:comparison-bound}.
Since $t\geq3/4$ in $D$ and $v^l$ is uniformly bounded,
\[
a^lt-r^lv^l\geq\tfrac34a_*-r^l\norm{v^l}_{L^\infty(D)}>0
\]
for large $l$. Also, \eqref{eq:comparison-M} and
\eqref{eq:comparison-bound} imply $\sigma_j(D^2v^l)>0$ for $j<k$.
The equation then implies $\sigma_k(D^2v^l)>0$, so
$D^2v^l\in\Gamma_k$.

Now let $z^l=U^l-v^l$. Uniform convergence to $Q$ implies
$\norm{z^l}_{L^\infty(D)}\to0$. Choose a small $\delta>0$, independent
of $l$, and, for $|s|\leq2\delta$, define
\begin{equation}\label{eq:savin-operator}
\begin{aligned}
f^l(s,X)&=\lambda(r^l)^k\{a^lt-r^l(v^l(X)+s)\}^k,\\
\mathcal F^l(M,s,X)&=\rho_{f^l(s,X)}(D^2v^l(X)+M).
\end{aligned}
\end{equation}
Extend the first line to all $s$ by replacing $s$ with a fixed smooth
bounded function that equals $s$ on $[-2\delta,2\delta]$. For large $l$,
the expression in braces remains positive, and this extension does not
change the equation for $z^l$. By construction,
\begin{equation}\label{eq:savin-zero}
\mathcal F^l(0,0,X)=0,\qquad
\mathcal F^l(D^2z^l,z^l,X)=0.
\end{equation}

We check that the structural constants in Savin's theorem are
independent of $l$. First, \eqref{eq:rho-properties} implies
\[
\mathcal F^l(M+P,s,X)\geq\mathcal F^l(M,s,X)
\qquad(M\in\Sym(n),\ P\geq0).
\]
This verifies Savin's global monotonicity hypothesis H1
\cite[p.~558]{Savin}, including outside the locally elliptic
neighborhood, and ensures that the classical solution $z^l$ is a
viscosity solution. Next,
\eqref{eq:comparison-bound} and \eqref{eq:rho-derivative} imply, for
$\delta$ small enough,
\[
cI\leq\left(\frac{\partial\mathcal F^l}{\partial M_{ij}}\right)
\leq CI
\qquad(\norm M\leq2\delta,\ |s|\leq2\delta,\ X\in D).
\]
For $\norm M,\norm P\leq\delta$ and $|s|\leq\delta$, the segment
from $M$ to $M+P$ stays in this neighborhood. Integrating along it,
we obtain
\[
c\sum_iP_{ii}
\leq\mathcal F^l(M+P,s,X)-\mathcal F^l(M,s,X)
\leq C\sum_iP_{ii}\qquad(P\geq0),
\]
which verifies H2. Finally, $\mathcal F^l$ is independent of the gradient
variable. The smoothness of $\rho_f$ near $(\bar M,0)$ and the common
$C^{4,\alpha}$ bound for $v^l$ imply
\[
\norm{\mathcal F^l}_{
C^2(\{\norm M\leq2\delta,\ |s|\leq2\delta\}\times\ov D)}
\leq C.
\]
Together with the exact zero-solution identity in
\eqref{eq:savin-zero}, this verifies Savin's hypothesis H4
\cite[p.~559]{Savin}, with bounds independent of $l$.

All hypotheses of Savin's theorem now hold with common constants.
After scaling $D$ to $B_1$, its regularity and quantitative approximation
results \cite[Theorem~1.3, p.~560; Proposition~4.1, p.~571]{Savin}
therefore yield, for all sufficiently large $l$,
\[
\norm{z^l}_{C^{2,\beta}(B_{1/8}(e_n))}\leq C
\]
for some $\beta>0$. Both the bound and the smallness threshold are
independent of $l$. For $K\Subset B_{1/8}(e_n)$, centered second
differences in unit directions therefore imply
\[
\norm{D^2z^l}_{L^\infty(K)}
\leq C_K\left(h^\beta+h^{-2}\norm{z^l}_{L^\infty(D)}\right)
\]
for all sufficiently small $h>0$. Letting first $l\to\infty$ and then
$h\downarrow0$, and using \eqref{eq:comparison-bound}, we conclude that
\[
\norm{D^2U^l-\bar M}_{L^\infty(K)}
\leq\norm{D^2z^l}_{L^\infty(K)}
+\norm{D^2v^l-\bar M}_{L^\infty(K)}\longrightarrow0.
\]
\end{proof}

We can now finish the contradiction argument from
Section~\ref{sec:boundary-dilation}. For large $l$, the rescaled domain
contains $\ov D$. Both parts of Theorem~\ref{thm:limits} apply and show
that
\[
U^l\longrightarrow Q_{B,b},\qquad
D^2U^l(e_n)\longrightarrow\bar M_b,\qquad
(F^{ij}(\bar M_b))>0.
\]
By orthogonal invariance, \eqref{eq:rescaled-hessian}, and continuity
of the Newton tensor,
\[
\begin{aligned}
0&=\lim_{l\to\infty}
\lambda_{\min}\bigl((F^{ij}(D^2u(x^l)))\bigr)\\
&=\lim_{l\to\infty}
\lambda_{\min}\bigl((F^{ij}(D^2U^l(e_n)))\bigr)\\
&=\lambda_{\min}\bigl((F^{ij}(\bar M_b))\bigr)>0,
\end{aligned}
\]
a contradiction. This proves \eqref{eq:gap} and completes the uniform
ellipticity part of Theorem~\ref{thm:main}.

\begin{remark}
To apply Fan's nonhomogeneous extension of Savin's small perturbation
theorem \cite[Theorem~1.7, p.~3]{Fan} directly to $U^l-Q$, replace
$v^l$ by $Q$ in \eqref{eq:savin-operator} and subtract
$\mathcal F^l(0,0,X)$ from the operator. The resulting source and
spatial oscillation have $C^\alpha$ norms $O((r^l)^k)$, while
\eqref{eq:rho-properties}--\eqref{eq:rho-derivative} provide the
structural bounds. The theorem therefore proves the same Hessian
convergence without the auxiliary solution $v^l$.
\end{remark}

\section{Completion of the proof of Theorem~\ref{thm:main}}\label{sec:regularity}

We first use \eqref{eq:gap} to write a concave, uniformly elliptic
equation for $u$. Boundary and interior estimates then yield
$C^{2,\alpha}$ regularity, and differentiation yields smoothness.

\subsection{A concave extension}

To combine boundary expansions with interior estimates, we construct
a concave operator that is uniformly elliptic on all symmetric matrices. We adapt the extension
by supporting affine functions used by Yu Wang
\cite[Definition~3.1 and Lemma~3.4, p.~944]{YuWang}.
Wang explicitly attributes this construction to Savin
\cite[Remark~3.3, p.~944]{YuWang}; see also
\cite[Lemma~4.1, pp.~447--448]{TWWY}.
Here we extend jointly in the Hessian and the function value. Let
\[
\mathcal K=\ov{\{(D^2u(x),u(x)):x\in\Omega\}}
\subset\Sym(n)\times[-1,0]
\]
and define
\begin{equation}\label{eq:concave-extension}
\widehat F(M,s)=\inf_{(\bar M,z)\in\mathcal K}
\left\{\sum_{i,j=1}^nF^{ij}(\bar M)M_{ij}
+k\lambda(-z)^{k-1}s\right\}.
\end{equation}
The $C^{1,1}$ bound makes $\mathcal K$ compact, and continuity extends
\eqref{eq:gap} to this set:
\[
\theta I\leq(F^{ij}(\bar M))\leq\Lambda I
\qquad((\bar M,z)\in\mathcal K).
\]
Hence, for $P\geq0$,
\[
\begin{gathered}
\theta\sum_iP_{ii}
\leq\widehat F(M+P,s)-\widehat F(M,s)
\leq\Lambda\sum_iP_{ii},\\
|\widehat F(M,s)-\widehat F(M,s')|\leq k\lambda|s-s'|.
\end{gathered}
\]
As the infimum of linear functions of $(M,s)$, $\widehat F$ is
jointly concave and positively homogeneous.

For $(\bar M,z)\in\mathcal K$, we have
$\sigma_k(\bar M)=\lambda(-z)^k$. Thus \eqref{eq:support}
and continuity imply
\[
\sum_{i,j=1}^nF^{ij}(\bar M)\bar M'_{ij}
+k\lambda(-z)^{k-1}z'\geq0
\qquad((\bar M,z),(\bar M',z')\in\mathcal K).
\]
When the two pairs coincide, Euler's identity shows that
\[
\sum_{i,j=1}^nF^{ij}(\bar M)\bar M_{ij}
+k\lambda(-z)^{k-1}z
=k\sigma_k(\bar M)-k\lambda(-z)^k=0.
\]
At $(D^2u(x),u(x))$, all expressions in the infimum are nonnegative
and one of them is zero. Consequently,
\begin{equation}\label{eq:extension-equation}
\widehat F(D^2u,u)=0\quad\text{in }\Omega.
\end{equation}

The boundary theorem used below allows gradient terms but no dependence
on the function value. We remove this dependence by adding one variable
$s\in(-1,1)$ and defining
\begin{equation}\label{eq:lift}
\begin{aligned}
w(x,s)&=e^su(x),\\
\mathcal G(M,P)&=\widehat F(M_{xx},P^{n+1})
+M_{n+1,n+1}-P^{n+1}.
\end{aligned}
\end{equation}
Here $M\in\Sym(n+1)$, $P\in\R^{n+1}$, and $M_{xx}$ is the
upper-left $n\times n$ block. Equivalently,
\begin{equation}\label{eq:Bellman}
\begin{aligned}
\mathcal G(M,P)
=\inf_{(\bar M,z)\in\mathcal K}
\biggl\{&\sum_{i,j=1}^nF^{ij}(\bar M)M_{ij}+M_{n+1,n+1}\\
&+\bigl(k\lambda(-z)^{k-1}-1\bigr)P^{n+1}\biggr\}.
\end{aligned}
\end{equation}
The coefficient matrices satisfy
\[
\min\{\theta,1\}I_{n+1}
\leq\diag\bigl((F^{ij}(\bar M)),1\bigr)
\leq\max\{\Lambda,1\}I_{n+1}.
\]
Thus $\mathcal G$ is uniformly elliptic, concave, Lipschitz in $P$,
and independent of position. Since
$D_x^2w=e^sD^2u$ and $\partial_sw=\partial_{ss}w=e^su$,
positive homogeneity implies
\[
\mathcal G(D^2w,Dw)=e^s\widehat F(D^2u,u)=0
\quad\text{in }\Omega\times(-1,1),
\qquad
w=0\quad\text{on }\partial\Omega\times(-1,1).
\]

\subsection{Boundary expansions and interior estimates}

We work near $\partial\Omega\times[-1/2,1/2]$, away from $s=\pm1$.
The coefficients in \eqref{eq:Bellman} are independent of position,
and both the source and boundary data are zero. The assumptions of
\cite[Theorem~1.2, p.~1696]{SS} therefore hold.
This theorem provides an exponent $\alpha_b>0$ and a H\"older
continuous matrix field $\bar M(Z)$ on the boundary. Since
$w\in C^{1,1}$, the linear coefficient at each boundary point $Z$
is $Dw(Z)$. The polynomial
\begin{equation}\label{eq:boundary-polynomial}
\begin{aligned}
P(Y)&=w(Z)+Dw(Z)\cdot(Y-Z)\\
&\quad+\frac12\sum_{i,j=1}^{n+1}
\bar M_{ij}(Z)(Y^i-Z^i)(Y^j-Z^j)
\end{aligned}
\end{equation}
satisfies
\begin{equation}\label{eq:boundary-expansion}
\begin{gathered}
|w(Y)-P(Y)|\leq C|Y-Z|^{2+\alpha_b},\\
\norm{\bar M(Z)-\bar M(Z')}\leq C|Z-Z'|^{\alpha_b},\qquad
\mathcal G(\bar M(Z),Dw(Z))=0.
\end{gathered}
\end{equation}
To pass from this expansion of $w$ to an estimate for $D^2w$,
we use interior estimates on balls whose radii are comparable
to the distance to the boundary.

Let $\alpha_i>0$ be the exponent in
\cite[Theorem~3.4, pp.~444--445]{TWWY}, applied with H\"older exponent
$1/2$ for the dependence on position, and choose
\begin{equation}\label{eq:alpha-choice}
0<\alpha<\min\{\alpha_b,\alpha_i,1/2\}.
\end{equation}
For an interior point $X$ near the boundary, let $Z$ be a nearest
boundary point and $d=|X-Z|$. On $B_{1/2}\subset\R^{n+1}$, define
\begin{equation}\label{eq:distance-scaling}
\begin{aligned}
v(Y)&=\frac{w(X+dY)-P(X+dY)}{d^{2+\alpha}},\\
\Phi(M,Y)&=d^{-\alpha}
\mathcal G\bigl(\bar M(Z)+d^\alpha M,Dw(X+dY)\bigr).
\end{aligned}
\end{equation}
The ball $B_{d/2}(X)$ lies inside the cylinder. Since
$|X+dY-Z|\leq3d/2$, \eqref{eq:boundary-expansion} implies
\[
\norm v_{L^\infty(B_{1/2})}\leq Cd^{\alpha_b-\alpha}\leq C,
\qquad \Phi(D^2v,Y)=0.
\]
The operator $\Phi$ has the same ellipticity constants and remains
concave in $M$. The $C^{1,1}$ bound for $w$, the Lipschitz dependence
of $\mathcal G$ on the gradient, and
$\mathcal G(\bar M(Z),Dw(Z))=0$ imply
\begin{equation}\label{eq:scaled-oscillation}
\begin{aligned}
|\Phi(0,Y)|
&\leq Cd^{-\alpha}|Dw(X+dY)-Dw(Z)|\leq Cd^{1-\alpha},\\
|\Phi(M,Y)-\Phi(M,Y')|
&\leq Cd^{-\alpha}|Dw(X+dY)-Dw(X+dY')|\\
&\leq Cd^{1-\alpha}|Y-Y'|.
\end{aligned}
\end{equation}
All constants are independent of $X$ and $d$. If $n+1$ is odd,
we apply the even-dimensional formulation in \cite{TWWY} after
adding one extra variable:
\[
\widetilde v(Y,\tau)=v(Y),\qquad
\widetilde\Phi(M,(Y,\tau))=\Phi(M_{YY},Y)+M_{\tau\tau}.
\]
This preserves the structural bounds. The interior estimate,
followed by restriction to $\tau=0$ when an extra variable is used,
therefore yields, in the original dimension,
\begin{equation}\label{eq:interior-scaled}
\norm{D^2v}_{L^\infty(B_{1/4})}
+[D^2v]_{C^{\alpha_i}(B_{1/4})}\leq C.
\end{equation}

Returning to the original scale, we obtain
\begin{equation}\label{eq:Hessian-boundary}
\norm{D^2w(X)-\bar M(Z)}
=d^\alpha\norm{D^2v(0)}\leq Cd^\alpha.
\end{equation}
It remains to compare the Hessians at two interior points $X,X'$.
Write $r=|X-X'|$. If $r\leq d/8$, both points lie in the ball
where \eqref{eq:interior-scaled} applies, and
\[
\begin{aligned}
\norm{D^2w(X)-D^2w(X')}
&=d^\alpha\norm{D^2v(0)-D^2v((X'-X)/d)}\\
&\leq Cd^{\alpha-\alpha_i}r^{\alpha_i}\leq Cr^\alpha.
\end{aligned}
\]
If $r>d/8$, let $Z'$ be a nearest boundary point to $X'$ and
$d'=|X'-Z'|$. The triangle inequality implies
$d+d'+|Z-Z'|\leq Cr$. Using \eqref{eq:Hessian-boundary} at both
points and the boundary estimate in \eqref{eq:boundary-expansion},
we find
\[
\begin{aligned}
\norm{D^2w(X)-D^2w(X')}
&\leq\norm{D^2w(X)-\bar M(Z)}
+\norm{\bar M(Z)-\bar M(Z')}\\
&\quad+\norm{\bar M(Z')-D^2w(X')}\\
&\leq C\bigl(d^\alpha+|Z-Z'|^{\alpha_b}+(d')^\alpha\bigr)
\leq Cr^\alpha.
\end{aligned}
\]
Thus $D^2w$ extends to the boundary with value $\bar M$ and is
$C^\alpha$ there. A finite boundary covering, together with interior
smoothness and restriction to $s=0$, now shows that
\begin{equation}\label{eq:C2alpha}
u\in C^{2,\alpha}(\ov\Omega).
\end{equation}

\subsection{Higher derivatives}

We first estimate derivatives with a tangential index. The equation
will then determine the remaining pure normal derivative.
Flatten a boundary patch by a smooth change of coordinates
$x=\Psi(y,t)$, and write $v=u\circ\Psi$ and $J=D\Psi$.
The Hessian of $u$ in these coordinates is
\begin{equation}\label{eq:flattened-Hessian}
M[v]=J^{-T}
\left(D^2v-\sum_{h=1}^n(J^{-T}Dv)^hD^2\Psi^h\right)J^{-1}.
\end{equation}
Hence
\[
\sigma_k(M[v])=\lambda(-v)^k,\qquad v(y,0)=0.
\]
By \eqref{eq:gap}, the principal coefficient matrix satisfies
\begin{equation}\label{eq:flattened-linearization}
cI\leq J^{-1}(F^{ij}(M[v]))J^{-T}\leq CI,
\end{equation}
while \eqref{eq:C2alpha} implies that its entries are $C^\alpha$.

For $a<n$, subtracting the equations at $(y+he_a,t)$ and $(y,t)$
and dividing by $h$ produces a linear equation for the tangential
difference quotient. Since $\Psi$ is smooth and $v\in C^{2,\alpha}$,
its coefficients and source have common $C^\alpha$ bounds.
Continuity of $v$, $Dv$, and $D^2v$, together with
\eqref{eq:flattened-linearization}, ensures that, for small $h$,
the intermediate values used in the linearization remain in a
fixed uniformly elliptic neighborhood.
The difference quotient vanishes on $t=0$ and satisfies
\[
\left\|\frac{v(\,\cdot+he_a)-v}{h}\right\|_{L^\infty}
\leq\norm{Dv}_{L^\infty}.
\]
Boundary Schauder estimates
\cite[Theorem~6.6, pp.~98--99; Theorem~6.19, pp.~111--112]{GT}
therefore imply, on a smaller patch,
\[
\left\|\frac{v(\,\cdot+he_a)-v}{h}\right\|_{C^{2,\alpha}}
\leq C.
\]
Passing to the limit as $h\to0$, we conclude that
$\partial_av\in C^{2,\alpha}$ for every $a<n$.
Thus all third derivatives with a tangential index are controlled.

To recover the pure normal derivative, let $\xi=J^{-T}e_n$ and
write $M[v]=M_0+\partial_{tt}v\,\xi\xi^T$. Here $M_0$ depends only
on the other second derivatives and on derivatives of lower order.
In the rank-one identity
\[
\sigma_k(M_0+s\xi\xi^T)
=\sigma_k(M_0)+s\sum_{i,j=1}^nF^{ij}(M_0)\xi^i\xi^j
\]
the coefficient of $s$ is independent of $s$. By
\eqref{eq:flattened-linearization}, its value at the actual Hessian
is positive, without any cone condition on $M_0$. We can therefore
solve for the pure normal second derivative:
\begin{equation}\label{eq:pure-normal}
\begin{gathered}
\partial_{tt}v=
\frac{\lambda(-v)^k-\sigma_k(M_0)}
{\sum_{i,j=1}^nF^{ij}(M_0)\xi^i\xi^j},\\
\sum_{i,j=1}^nF^{ij}(M_0)\xi^i\xi^j
=\sum_{i,j=1}^nF^{ij}(M[v])\xi^i\xi^j\geq c>0.
\end{gathered}
\end{equation}
When the right-hand side is differentiated once in $t$, every
third derivative of $v$ that occurs has a tangential index.
The preceding estimates therefore imply
$\partial_{ttt}v\in C^\alpha$, and hence $v\in C^{3,\alpha}$.

The same argument applies at every order. Suppose that
$v\in C^{m,\alpha}$, with $m\geq2$.
The tangentially differentiated equations have coefficients and
sources in $C^{m-2,\alpha}$. Boundary Schauder estimates imply
\[
\partial_av\in C^{m,\alpha}\qquad(a<n),
\]
so every derivative of order $m+1$ with a tangential index is
$C^\alpha$. When we differentiate \eqref{eq:pure-normal} $m-1$
times in $t$, every derivative of order $m+1$ on the right-hand
side has a tangential index. The induction hypothesis controls all lower
derivatives. Therefore,
\[
\partial_t^{m+1}v
=\partial_t^{m-1}
\left(\frac{\lambda(-v)^k-\sigma_k(M_0)}
{\sum_{i,j=1}^nF^{ij}(M_0)\xi^i\xi^j}\right)\in C^\alpha.
\]
We have proved
$v\in C^{m,\alpha}\Longrightarrow v\in C^{m+1,\alpha}$.
A finite covering of the boundary, together with interior smoothness,
shows that $u\in C^\infty(\ov\Omega)$. This and \eqref{eq:gap}
complete the proof of Theorem~\ref{thm:main}.\qed

The logarithmic transform in \cite{LMQS} also satisfies a
quantitative convexity estimate near the boundary.

\begin{corollary}\label{cor:log-collar}
Under the assumptions of Theorem~\ref{thm:main}, let
$v=-\log(-u)$ and $d(x)=\dist(x,\partial\Omega)$.
There exist $c,\delta>0$ such that
\[
\sum_{i,j=1}^n \partial_{ij}v(x)\xi^i\xi^j
\geq \frac{c}{d(x)}|\xi|^2
\qquad(0<d(x)<\delta,\ \xi\in\R^n).
\]
In particular, $D^2v$ is positive definite near the boundary.
\end{corollary}

\begin{proof}
Since $u$ is smooth and vanishes on the boundary, $-u/d$ extends
smoothly to a boundary neighborhood. By the Hopf lemma, its boundary
value is $u_\nu>0$. Consequently,
\[
-u/d\in C^\infty,\qquad -u/d\geq c>0,\qquad
\norm{D^2\log(-u/d)}\leq C.
\]
Write $v=-\log d-\log(-u/d)$. Let $q$ be the nearest boundary
point and $\kappa_i(q)>0$ its principal curvatures.
In the corresponding principal frame, with the last vector normal
to the boundary, the uniform convexity of $\partial\Omega$ yields,
for sufficiently small $d>0$,
\[
\begin{aligned}
\sum_{i,j=1}^n\partial_{ij}v(x)\xi^i\xi^j
&=\sum_{i=1}^{n-1}
\frac{\kappa_i(q)(\xi^i)^2}{d(1-d\kappa_i(q))}
+\frac{(\xi^n)^2}{d^2}\\
&\quad-\sum_{i,j=1}^n\partial_{ij}\log(-u/d)(x)\xi^i\xi^j\\
&\geq\frac c d|\xi|^2-C|\xi|^2
\geq\frac{c}{2d}|\xi|^2
\end{aligned}
\]
\end{proof}

\begin{remark}\label{rem:MA-endpoint}
For $k=n$, the quadratic family has
$c_b=\sum_{i,j=1}^{n-1}(B^{-1})_{ij}b^ib^j$.
Its Hessian is positive semidefinite of rank $n-1$, and its
cofactor matrix has rank one. Thus \eqref{eq:gap} fails at this
endpoint. The lower bound in \eqref{eq:rank-bound} becomes one,
so the tangential vector used in Proposition~\ref{prop:coercivity}
need not exist. At this endpoint, Le and Savin's degenerate boundary
Schauder estimates and transformation argument prove global smoothness
of the Monge--Amp\`ere eigenfunction \cite[Theorem~1.4]{LS}.
For $k=1$, the equation is the linear Dirichlet eigenvalue equation.
\end{remark}

\section*{Acknowledgments}

The authors acknowledge support from Grant 2025YFA1017603 of the
National Key R\&D Program of China and Grant 12571227 of the National
Natural Science Foundation of China.

\section*{Declaration on the use of artificial intelligence}

Examples~\ref{ex:zero-level}, \ref{ex:newton-degeneracy} and
\ref{ex:convex-kernel} show that, when $k<n$, the
$k$-Hessian equation can remain elliptic when the right-hand side is zero,
and that the Hessian kernels required for degeneracy are incompatible
with positive definite tangential boundary data. These observations
provided the main motivation for our theorem and led us to expect
uniform ellipticity and smoothness up to the boundary for eigenfunctions
on uniformly convex domains. Through extensive
discussions with GPT-6 Astra over about a month, we resolved the
analytic difficulties described in the introduction. Most of the
blow-up techniques used in the proof were suggested by GPT.
We have verified the proofs and take full responsibility for the
mathematical content of the paper.

\begingroup
\small
\raggedright
\bibliographystyle{amsplain}
\bibliography{references}

@article{Fan,
  author  = {Fan, Z.},
  title   = {A generalization of {Savin}'s small perturbation theorem for fully nonlinear elliptic equations and applications},
  journal = {Calc. Var. Partial Differential Equations},
  volume  = {65},
  number  = {6},
  year    = {2026},
  pages   = {Paper No.~182},
  doi     = {10.1007/s00526-026-03349-7},
  url     = {https://doi.org/10.1007/s00526-026-03349-7},
  note    = {\href{https://doi.org/10.1007/s00526-026-03349-7}{doi:\,\nolinkurl{10.1007/s00526-026-03349-7}}}
}

@book{GT,
  author    = {Gilbarg, D. and Trudinger, N. S.},
  title     = {Elliptic Partial Differential Equations of Second Order},
  series    = {Classics in Mathematics},
  publisher = {Springer-Verlag},
  address   = {Berlin},
  year      = {2001},
  note      = {Reprint of the 1998 edition; \href{https://doi.org/10.1007/978-3-642-61798-0}{doi:\,\nolinkurl{10.1007/978-3-642-61798-0}}},
  doi       = {10.1007/978-3-642-61798-0},
  url       = {https://doi.org/10.1007/978-3-642-61798-0},
  edition   = {2nd},
  isbn      = {978-3-540-41160-4},
}

@misc{HeHuang,
  author        = {He, R. and Huang, G.},
  title         = {Weighted eigenvalue problem for a class of {Hessian} equations},
  year          = {2025},
  howpublished  = {Preprint, \href{https://arxiv.org/abs/2505.03231v1}{arXiv:2505.03231v1}},
  eprint        = {2505.03231},
  archivePrefix = {arXiv},
  version       = {1},
  url           = {https://arxiv.org/abs/2505.03231v1},
  doi           = {10.48550/arXiv.2505.03231},
  primaryClass  = {math.AP},
  note          = {A revised version is to appear in Indiana Univ. Math. J.; \href{https://doi.org/10.48550/arXiv.2505.03231}{arXiv doi:\,\nolinkurl{10.48550/arXiv.2505.03231}}},
}

@article{HHW,
  author  = {Hong, J. and Huang, G. and Wang, W.},
  title   = {Existence of global smooth solutions to {Dirichlet} problem for degenerate elliptic {Monge--Amp\`ere} equations},
  journal = {Comm. Partial Differential Equations},
  volume  = {36},
  number  = {4},
  year    = {2011},
  pages   = {635--656},
  doi     = {10.1080/03605302.2010.514171},
  url     = {https://doi.org/10.1080/03605302.2010.514171},
  note    = {\href{https://doi.org/10.1080/03605302.2010.514171}{doi:\,\nolinkurl{10.1080/03605302.2010.514171}}}
}

@article{ITW,
  author  = {Ivochkina, N. M. and Trudinger, N. S. and Wang, X.-J.},
  title   = {The {Dirichlet} problem for degenerate {Hessian} equations},
  journal = {Comm. Partial Differential Equations},
  volume  = {29},
  number  = {1--2},
  year    = {2004},
  pages   = {219--235},
  doi     = {10.1081/PDE-120028851},
  url     = {https://doi.org/10.1081/PDE-120028851},
  note    = {\href{https://doi.org/10.1081/PDE-120028851}{doi:\,\nolinkurl{10.1081/PDE-120028851}}}
}

@article{JW,
  author  = {Jiao, H. and Wang, Z.},
  title   = {Second order estimates for convex solutions of degenerate {$k$}-{Hessian} equations},
  journal = {J. Funct. Anal.},
  volume  = {286},
  number  = {3},
  year    = {2024},
  pages   = {Paper No.~110248},
  doi     = {10.1016/j.jfa.2023.110248},
  url     = {https://doi.org/10.1016/j.jfa.2023.110248},
  note    = {\href{https://doi.org/10.1016/j.jfa.2023.110248}{doi:\,\nolinkurl{10.1016/j.jfa.2023.110248}}}
}

@article{Le,
  author  = {Le, N. Q.},
  title   = {A spectral characterization and an approximation scheme for the {Hessian} eigenvalue},
  journal = {Rev. Mat. Iberoam.},
  volume  = {38},
  number  = {5},
  year    = {2022},
  pages   = {1473--1500},
  doi     = {10.4171/RMI/1306},
  url     = {https://doi.org/10.4171/RMI/1306},
  note    = {\href{https://doi.org/10.4171/RMI/1306}{doi:\,\nolinkurl{10.4171/RMI/1306}}}
}

@misc{LeICBS,
  author       = {Le, N. Q.},
  title        = {The {Monge--Amp\`ere} eigenvalue problem and global smoothness of the eigenfunctions},
  year         = {2024},
  howpublished = {Lecture slides, International Congress of Basic Science, Beijing, July 16},
  url          = {https://lsp.icbs.cn/upload/156-1721193461-A14-304-Nam%20Quang%20Le-0716.pdf},
  note         = {\href{https://lsp.icbs.cn/upload/156-1721193461-A14-304-Nam\%20Quang\%20Le-0716.pdf}{Public lecture slides}}
}

@article{LS,
  author  = {Le, N. Q. and Savin, O.},
  title   = {{Schauder} estimates for degenerate {Monge--Amp\`ere} equations and smoothness of the eigenfunctions},
  journal = {Invent. Math.},
  volume  = {207},
  number  = {1},
  year    = {2017},
  pages   = {389--423},
  doi     = {10.1007/s00222-016-0677-1},
  url     = {https://doi.org/10.1007/s00222-016-0677-1},
  note    = {\href{https://doi.org/10.1007/s00222-016-0677-1}{doi:\,\nolinkurl{10.1007/s00222-016-0677-1}}},
}

@misc{LMQS,
  author        = {Li, J. and Ma, X.-N. and Qiu, G. and Salani, P.},
  title         = {A {Brunn--Minkowski} inequality and convexity for the {$2$}-{Hessian} eigenvalue in convex domains},
  year          = {2026},
  howpublished  = {Preprint, \href{https://arxiv.org/abs/2609.09175v1}{arXiv:2609.09175v1}},
  eprint        = {2609.09175},
  archivePrefix = {arXiv},
  version       = {1},
  url           = {https://arxiv.org/abs/2609.09175v1},
  doi           = {10.48550/arXiv.2609.09175},
  primaryClass  = {math.AP},
  note          = {\href{https://doi.org/10.48550/arXiv.2609.09175}{arXiv doi:\,\nolinkurl{10.48550/arXiv.2609.09175}}},
}

@article{Lions,
  author  = {Lions, P.-L.},
  title   = {Two remarks on {Monge--Amp\`ere} equations},
  journal = {Ann. Mat. Pura Appl. (4)},
  volume  = {142},
  year    = {1986},
  pages   = {263--275},
  doi     = {10.1007/BF01766596},
  url     = {https://doi.org/10.1007/BF01766596},
  note    = {\href{https://doi.org/10.1007/BF01766596}{doi:\,\nolinkurl{10.1007/BF01766596}}},
}

@article{LMX,
  author  = {Liu, P. and Ma, X.-N. and Xu, L.},
  title   = {A {Brunn--Minkowski} inequality for the {Hessian} eigenvalue in three-dimensional convex domain},
  journal = {Adv. Math.},
  volume  = {225},
  number  = {3},
  year    = {2010},
  pages   = {1616--1633},
  doi     = {10.1016/j.aim.2010.04.003},
  url     = {https://doi.org/10.1016/j.aim.2010.04.003},
  note    = {\href{https://doi.org/10.1016/j.aim.2010.04.003}{doi:\,\nolinkurl{10.1016/j.aim.2010.04.003}}}
}

@misc{QZ,
  author        = {Qiu, G. and Zhou, X.},
  title         = {A priori interior estimates for special {Lagrangian} curvature equations},
  year          = {2024},
  howpublished  = {Preprint, \href{https://arxiv.org/abs/2407.15159v1}{arXiv:2407.15159v1}},
  eprint        = {2407.15159},
  archivePrefix = {arXiv},
  version       = {1},
  url           = {https://arxiv.org/abs/2407.15159v1},
  doi           = {10.48550/arXiv.2407.15159},
  primaryClass  = {math.AP},
  note          = {\href{https://doi.org/10.48550/arXiv.2407.15159}{arXiv doi:\,\nolinkurl{10.48550/arXiv.2407.15159}}},
}

@article{SalaniMA,
  author  = {Salani, P.},
  title   = {A {Brunn--Minkowski} inequality for the {Monge--Amp\`ere} eigenvalue},
  journal = {Adv. Math.},
  volume  = {194},
  number  = {1},
  year    = {2005},
  pages   = {67--86},
  doi     = {10.1016/j.aim.2004.05.011},
  url     = {https://doi.org/10.1016/j.aim.2004.05.011},
  note    = {\href{https://doi.org/10.1016/j.aim.2004.05.011}{doi:\,\nolinkurl{10.1016/j.aim.2004.05.011}}}
}

@article{SalaniHessian,
  author        = {Salani, P.},
  title         = {Convexity of solutions and {Brunn--Minkowski} inequalities for {Hessian} equations in {$\mathbb R^3$}},
  journal       = {Adv. Math.},
  volume        = {229},
  number        = {3},
  year          = {2012},
  pages         = {1924--1948},
  eprint        = {1101.0496},
  archivePrefix = {arXiv},
  url           = {https://doi.org/10.1016/j.aim.2011.12.009},
  note          = {\href{https://doi.org/10.1016/j.aim.2011.12.009}{doi:\,\nolinkurl{10.1016/j.aim.2011.12.009}}; \href{https://arxiv.org/abs/1101.0496}{arXiv:1101.0496}},
  doi           = {10.1016/j.aim.2011.12.009},
}

@article{Savin,
  author  = {Savin, O.},
  title   = {Small perturbation solutions for elliptic equations},
  journal = {Comm. Partial Differential Equations},
  volume  = {32},
  number  = {4--6},
  year    = {2007},
  pages   = {557--578},
  doi     = {10.1080/03605300500394405},
  url     = {https://doi.org/10.1080/03605300500394405},
  note    = {\href{https://doi.org/10.1080/03605300500394405}{doi:\,\nolinkurl{10.1080/03605300500394405}}}
}

@article{SavinBoundary,
  author  = {Savin, O.},
  title   = {A localization theorem and boundary regularity for a class of degenerate {Monge--Amp\`ere} equations},
  journal = {J. Differential Equations},
  volume  = {256},
  number  = {2},
  year    = {2014},
  pages   = {327--388},
  doi     = {10.1016/j.jde.2013.08.019},
  url     = {https://doi.org/10.1016/j.jde.2013.08.019},
  note    = {\href{https://doi.org/10.1016/j.jde.2013.08.019}{doi:\,\nolinkurl{10.1016/j.jde.2013.08.019}}},
}

@article{SS,
  author  = {Silvestre, L. and Sirakov, B.},
  title   = {Boundary regularity for viscosity solutions of fully nonlinear elliptic equations},
  journal = {Comm. Partial Differential Equations},
  volume  = {39},
  number  = {9},
  year    = {2014},
  pages   = {1694--1717},
  doi     = {10.1080/03605302.2013.842249},
  url     = {https://doi.org/10.1080/03605302.2013.842249},
  note    = {\href{https://doi.org/10.1080/03605302.2013.842249}{doi:\,\nolinkurl{10.1080/03605302.2013.842249}}}
}

@article{TWX,
  author  = {Tian, G. and Wang, Q. and Xu, C.-J.},
  title   = {Local solvability of the {$k$}-{Hessian} equations},
  journal = {Sci. China Math.},
  volume  = {59},
  number  = {9},
  year    = {2016},
  pages   = {1753--1768},
  doi     = {10.1007/s11425-016-5135-4},
  url     = {https://doi.org/10.1007/s11425-016-5135-4},
  note    = {\href{https://doi.org/10.1007/s11425-016-5135-4}{doi:\,\nolinkurl{10.1007/s11425-016-5135-4}}}
}

@article{TWWY,
  author  = {Tosatti, V. and Wang, Y. and Weinkove, B. and Yang, X.},
  title   = {{$C^{2,\alpha}$} estimates for nonlinear elliptic equations in complex and almost complex geometry},
  journal = {Calc. Var. Partial Differential Equations},
  volume  = {54},
  number  = {1},
  year    = {2015},
  pages   = {431--453},
  doi     = {10.1007/s00526-014-0791-0},
  url     = {https://doi.org/10.1007/s00526-014-0791-0},
  note    = {\href{https://doi.org/10.1007/s00526-014-0791-0}{doi:\,\nolinkurl{10.1007/s00526-014-0791-0}}}
}

@article{Tso,
  author  = {Tso, K.},
  title   = {On a real {Monge--Amp\`ere} functional},
  journal = {Invent. Math.},
  volume  = {101},
  number  = {2},
  year    = {1990},
  pages   = {425--448},
  doi     = {10.1007/BF01231510},
  url     = {https://doi.org/10.1007/BF01231510},
  note    = {\href{https://doi.org/10.1007/BF01231510}{doi:\,\nolinkurl{10.1007/BF01231510}}},
}

@article{Wang94,
  author  = {Wang, X.-J.},
  title   = {A class of fully nonlinear elliptic equations and related functionals},
  journal = {Indiana Univ. Math. J.},
  volume  = {43},
  number  = {1},
  year    = {1994},
  pages   = {25--54},
  doi     = {10.1512/iumj.1994.43.43002},
  url     = {https://doi.org/10.1512/iumj.1994.43.43002},
  note    = {\href{https://doi.org/10.1512/iumj.1994.43.43002}{doi:\,\nolinkurl{10.1512/iumj.1994.43.43002}}},
  issn    = {0022-2518},
}

@incollection{Wang09,
  author    = {Wang, X.-J.},
  title     = {The {$k$}-{Hessian} equation},
  booktitle = {Geometric Analysis and PDEs},
  series    = {Lecture Notes in Math.},
  volume    = {1977},
  publisher = {Springer},
  address   = {Dordrecht},
  year      = {2009},
  pages     = {177--252},
  doi       = {10.1007/978-3-642-01674-5_5},
  url       = {https://doi.org/10.1007/978-3-642-01674-5_5},
  note      = {\href{https://doi.org/10.1007/978-3-642-01674-5_5}{doi:\,\nolinkurl{10.1007/978-3-642-01674-5_5}}},
  editor    = {Ambrosetti, A. and Chang, S.-Y. A. and Malchiodi, A.},
  isbn      = {978-3-642-01673-8},
}

@article{WX,
  author  = {Wang, Q. and Xu, C.-J.},
  title   = {{$C^{1,1}$} solution of the {Dirichlet} problem for degenerate {$k$}-{Hessian} equations},
  journal = {Nonlinear Anal.},
  volume  = {104},
  year    = {2014},
  pages   = {133--146},
  doi     = {10.1016/j.na.2014.03.016},
  url     = {https://doi.org/10.1016/j.na.2014.03.016},
  note    = {\href{https://doi.org/10.1016/j.na.2014.03.016}{doi:\,\nolinkurl{10.1016/j.na.2014.03.016}}}
}

@misc{Zhang,
  author        = {Zhang, J.},
  title         = {The {Dirichlet} eigenvalue problems for some concave elliptic {Hessian} operators},
  year          = {2025},
  howpublished  = {Preprint, \href{https://arxiv.org/abs/2510.16748v3}{arXiv:2510.16748v3}},
  eprint        = {2510.16748},
  archivePrefix = {arXiv},
  version       = {3},
  url           = {https://arxiv.org/abs/2510.16748v3},
  doi           = {10.48550/arXiv.2510.16748},
  primaryClass  = {math.AP},
  note          = {\href{https://doi.org/10.48550/arXiv.2510.16748}{arXiv doi:\,\nolinkurl{10.48550/arXiv.2510.16748}}},
}

@article{CIL,
  author  = {Crandall, M. G. and Ishii, H. and Lions, P.-L.},
  title   = {User's guide to viscosity solutions of second order partial differential equations},
  journal = {Bull. Amer. Math. Soc. (N.S.)},
  volume  = {27},
  number  = {1},
  year    = {1992},
  pages   = {1--67},
  doi     = {10.1090/S0273-0979-1992-00266-5},
  url     = {https://doi.org/10.1090/S0273-0979-1992-00266-5},
  note    = {\href{https://doi.org/10.1090/S0273-0979-1992-00266-5}{doi:\,\nolinkurl{10.1090/S0273-0979-1992-00266-5}}}
}

@book{CHLP,
  author    = {Cirant, M. and Harvey, F. R. and Lawson, Jr., H. B. and Payne, K. R.},
  title     = {Comparison Principles for General Potential Theories and {PDEs}},
  series    = {Annals of Mathematics Studies},
  volume    = {218},
  publisher = {Princeton University Press},
  address   = {Princeton, NJ},
  year      = {2023},
  isbn      = {978-0-691-24361-0},
  doi       = {10.1515/9780691243641},
  url       = {https://doi.org/10.1515/9780691243641},
  note      = {\href{https://doi.org/10.1515/9780691243641}{doi:\,\nolinkurl{10.1515/9780691243641}}}
}

@article{DeSilvaSavin,
  author  = {De Silva, D. and Savin, O.},
  title   = {A note on higher regularity boundary {Harnack} inequality},
  journal = {Discrete Contin. Dyn. Syst.},
  volume  = {35},
  number  = {12},
  year    = {2015},
  pages   = {6155--6163},
  doi     = {10.3934/dcds.2015.35.6155},
  url     = {https://doi.org/10.3934/dcds.2015.35.6155},
  note    = {\href{https://doi.org/10.3934/dcds.2015.35.6155}{doi:\,\nolinkurl{10.3934/dcds.2015.35.6155}}}
}

@article{IPY,
  author  = {Ivochkina, N. M. and Prokof'eva, S. I. and Yakunina, G. V.},
  title   = {The {G\aa rding} cones in the modern theory of fully nonlinear second order differential equations},
  journal = {J. Math. Sci.},
  volume  = {184},
  number  = {3},
  year    = {2012},
  pages   = {295--315},
  doi     = {10.1007/s10958-012-0869-1},
  url     = {https://doi.org/10.1007/s10958-012-0869-1},
  note    = {\href{https://doi.org/10.1007/s10958-012-0869-1}{doi:\,\nolinkurl{10.1007/s10958-012-0869-1}}}
}

@article{JiaMa,
  author  = {Jia, X. and Ma, S.},
  title   = {The {Liouville} theorem for {$k$}-{Hessian} equations in the half space},
  journal = {Calc. Var. Partial Differential Equations},
  volume  = {64},
  number  = {6},
  year    = {2025},
  pages   = {Paper No.~192},
  doi     = {10.1007/s00526-025-03059-6},
  url     = {https://doi.org/10.1007/s00526-025-03059-6},
  note    = {\href{https://doi.org/10.1007/s00526-025-03059-6}{doi:\,\nolinkurl{10.1007/s00526-025-03059-6}}}
}

@article{TTV,
  author  = {Terracini, S. and Tortone, G. and Vita, S.},
  title   = {Higher order boundary {Harnack} principle via degenerate equations},
  journal = {Arch. Ration. Mech. Anal.},
  volume  = {248},
  number  = {2},
  year    = {2024},
  pages   = {Paper No.~29},
  doi     = {10.1007/s00205-024-01973-1},
  url     = {https://doi.org/10.1007/s00205-024-01973-1},
  note    = {44 pp.; \href{https://doi.org/10.1007/s00205-024-01973-1}{doi:\,\nolinkurl{10.1007/s00205-024-01973-1}}}
}

@article{YuWang,
  author  = {Wang, Y.},
  title   = {On the {$C^{2,\alpha}$}-regularity of the complex {Monge--Amp\`ere} equation},
  journal = {Math. Res. Lett.},
  volume  = {19},
  number  = {4},
  year    = {2012},
  pages   = {939--946},
  doi     = {10.4310/MRL.2012.v19.n4.a16},
  url     = {https://doi.org/10.4310/MRL.2012.v19.n4.a16},
  note    = {\href{https://doi.org/10.4310/MRL.2012.v19.n4.a16}{doi:\,\nolinkurl{10.4310/MRL.2012.v19.n4.a16}}}
}

@misc{ZZ,
  author        = {Zhou, X. and Zhu, R.},
  title         = {Rigidity for convex entire solutions to the {$k$}-{Hessian} equation},
  year          = {2026},
  howpublished  = {Preprint, \href{https://arxiv.org/abs/2609.20445v1}{arXiv:2609.20445v1}},
  eprint        = {2609.20445},
  archivePrefix = {arXiv},
  version       = {1},
  primaryClass  = {math.AP},
  url           = {https://arxiv.org/abs/2609.20445v1},
  doi           = {10.48550/arXiv.2609.20445},
  note          = {\href{https://doi.org/10.48550/arXiv.2609.20445}{arXiv doi:\,\nolinkurl{10.48550/arXiv.2609.20445}}}
}
\endgroup
\printauthorinformation

\end{document}